\documentclass[12pt]{amsart}       
\usepackage{txfonts}
\usepackage{extarrows}
\usepackage{amssymb}
\usepackage{eucal}
\usepackage{bbm}
\usepackage{graphicx}
\usepackage{amsmath}
\usepackage{amscd}
\usepackage[all]{xy}           
\usepackage{tikz}
\usepackage{amsfonts,latexsym}
\usepackage{xspace}
\usepackage{epsfig}
\usepackage{float}
\usepackage{color}
\usepackage{shuffle}
\usepackage{fancybox}
\usepackage{colordvi}
\usepackage{multicol}
\usepackage{mathrsfs}
\usepackage{ifpdf}
\ifpdf
  \usepackage[colorlinks,final,hyperindex,backref=page]{hyperref}
\else
  \usepackage[colorlinks,final,hyperindex,backref=page,hypertex]{hyperref}
\fi

\renewcommand*{\backref}[1]{}
\renewcommand*{\backrefalt}[4]{%
  \ifcase #1%
  \or \textcolor{red}{#2}%
  \else \textcolor{red}{#2}%
  \fi}

\newtheorem{theorem}{Theorem}[section]
\newtheorem{proposition}[theorem]{Proposition}
\newtheorem{lemma}[theorem]{Lemma}
\newtheorem{coro}[theorem]{Corollary}
\newtheorem{corollary}[theorem]{Corollary}
\newtheorem{prop-def}{Proposition-Definition}[section]
\newtheorem{coro-def}{Corollary-Definition}[section]

\theoremstyle{definition}
\newtheorem{definition}[theorem]{Definition}
\newtheorem{remark}[theorem]{Remark}

\newcommand{\nc}{\newcommand}
\nc{\tred}[1]{\textcolor{red}{#1}}
\nc{\tblue}[1]{\textcolor{blue}{#1}}
\nc{\tgreen}[1]{\textcolor{green}{#1}}
\nc{\tpurple}[1]{\textcolor{purple}{#1}}
\nc{\btred}[1]{\textcolor{red}{\bf #1}}
\nc{\btblue}[1]{\textcolor{blue}{\bf #1}}
\nc{\btgreen}[1]{\textcolor{green}{\bf #1}}
\nc{\btpurple}[1]{\textcolor{purple}{\bf #1}}
\nc{\NN}{{\mathbb N}}
\nc{\ncsha}{{\mbox{\cyr X}^{\mathrm NC}}} \nc{\ncshao}{{\mbox{\cyr
X}^{\mathrm NC}_0}}

\newcommand{\delete}[1]{}

\nc{\mlabel}[1]{\label{#1}}
\nc{\mcite}[1]{\cite{#1}}
\nc{\mref}[1]{\ref{#1}}
\nc{\meqref}[1]{\eqref{#1}}
\nc{\mbibitem}[1]{\bibitem{#1}}
\nc{\sha}{{\mbox{\cyr X}}}  
\newfont{\scyr}{wncyr10 scaled 550}
\nc{\ssha}{\mbox{\bf \scyr X}}
\nc{\shap}{{\mbox{\cyrs X}}} 
\nc{\shpr}{\diamond}    
\nc{\shp}{\ast} \nc{\shplus}{\shpr^+}
\nc{\shprc}{\shpr_c}    
\nc{\dep}{\mrm{dep}} \nc{\lc}{\lfloor} \nc{\rc}{\rfloor}
\nc{\db}{\leq_{\rm db}} \nc{\bfk}{{\bf k}}

\nc{\cala}{{\mathcal A}} \nc{\calb}{{\mathcal B}}
\nc{\calc}{{\mathcal C}}
\nc{\cald}{{\mathcal D}} \nc{\cale}{{\mathcal E}}
\nc{\calf}{{\mathcal F}} \nc{\calg}{{\mathcal G}}
\nc{\calh}{{\mathcal H}} \nc{\cali}{{\mathcal I}}
\nc{\call}{{\mathcal L}} \nc{\calm}{{\mathcal M}}
\nc{\caln}{{\mathcal N}} \nc{\calo}{{\mathcal O}}
\nc{\calp}{{\mathcal P}} \nc{\calr}{{\mathcal R}}
\nc{\cals}{{\mathcal S}} \nc{\calt}{{\mathcal T}}
\nc{\calu}{{\mathcal U}} \nc{\calw}{{\mathcal W}} \nc{\calk}{{\mathcal K}}
\nc{\calx}{{\mathcal X}} \nc{\CA}{\mathcal{A}}

\nc{\fraka}{{\mathfrak a}} \nc{\frakA}{{\mathfrak A}}
\nc{\frakb}{{\mathfrak b}} \nc{\frakB}{{\mathfrak B}}
\nc{\frakc}{{\mathfrak c}}
\nc{\frakD}{{\mathfrak D}} \nc{\frakF}{\mathfrak{F}}
\nc{\frakf}{{\mathfrak f}} \nc{\frakg}{{\mathfrak g}}
\nc{\frakH}{{\mathfrak H}} \nc{\frakL}{{\mathfrak L}}
\nc{\frakM}{{\mathfrak M}} \nc{\bfrakM}{\overline{\frakM}}
\nc{\frakm}{{\mathfrak m}} \nc{\frakP}{{\mathfrak P}}
\nc{\frakN}{{\mathfrak N}} \nc{\frakp}{{\mathfrak p}}
\nc{\frakS}{{\mathfrak S}} \nc{\frakT}{\mathfrak{T}}
\nc{\frakX}{{\mathfrak X}}

\font\cyr=wncyr10 \font\cyrs=wncyr7
\nc{\li}[1]{\textcolor{blue}{Nan:#1}}
\nc{\lir}[1]{\textcolor{red}{Li:#1}}
\nc{\yi}[1]{\textcolor{blue}{Yi: #1}}
\nc{\xing}[1]{\textcolor{purple}{Xing:#1}}
\nc{\revise}[1]{\textcolor{red}{#1}}
\nc{\nan}[1]{\textcolor{blue}{Nan:#1}}

\numberwithin{equation}{section}
\nc{\RR}{\mathbb{R}}
\nc{\X}{{\bf X}}
\nc{\E}{{\bf E}}
\nc{\x}{\mathbb{X}}
\nc{\C}{\mathcal{C}^{\alpha}}
\nc{\D}{\mathcal{D}^{\alpha}}
\nc{\CC}{\mathcal{C}_{\X}^{\alpha}}
\nc{\f}{\varphi}
\nc{\al}{\alpha}
\nc{\lbar}{\overline}
\nc{\HA}{\mathbb{S}}
\nc{\ha}{\mathcal{S}}

\nc{\V}{V} \nc{\pro}{\otimes}
\nc{\tng}{T^{\le N}(V)^{g}} \nc{\tn}{T^{\le N}(V)}
\nc{\ttg}{T^{\le 3}(V)^{g}}
\nc{\ZZ}{\mathbb{Z}} \nc{\etree}{1}
\nc{\xx}{\mathcal{X}}
\nc{\RP}{{\mathcal{D}}^{\alpha}([0, T]^2, V)}
\nc{\Y}{{\bf Y}}\nc{\id}{\text{id}} \nc{\Id}{\text{Id}}\nc{\Z}{{\bf Z}}
\nc{\sym}{\operatorname{Sym}} \nc{\tri}{\operatorname{Sym}}
\providecommand{\hC}{\hat C}
\providecommand{\hdd}{\hat\delta}
\providecommand{\R}{\mathbb R}
\providecommand{\T}{\mathbb T}
\providecommand{\norm}[1]{\left\|#1\right\|}

\renewcommand{\D}{\mathcal D}
\renewcommand{\Z}{\mathcal Z}

\begin{document}

\title[A low-regularity semigroup Sewing lemma]{A Low-Regularity Semigroup Sewing Lemma via Quotient Structures}
%
%
\author{Nannan Li}
\address{School of Mathematics and Statistics, Lanzhou University
Lanzhou, 730000, China
}
\email{linn2024@lzu.edu.cn}

\author{Xing Gao$^{*}$}\thanks{*Corresponding author}
\address{School of Mathematics and Statistics, Lanzhou University
Lanzhou, 730000, China; Gansu Provincial Research Center for Basic Disciplines of Mathematics
and Statistics, Lanzhou, 730070, China
}
\email{gaoxing@lzu.edu.cn}
\begin{abstract}
We develop a low-regularity Sewing theory for the semigroup coboundary $\hat\delta=\delta-a$ associated with a strongly continuous semigroup $S$. Unlike the ordinary low-regularity Sewing problem, the semigroup setting has an intrinsic algebraic non-uniqueness below the threshold $1$, in the sense that solutions are canonical only modulo semigroup cocycles.  Accordingly, the natural target is a quotient space rather than an increment space.

We identify this quotient structure and construct the corresponding semigroup Sewing map.  The construction uses a frozen terminal-time transform, which rewrites semigroup defects, for each terminal time, as ordinary low-regularity Sewing problems on a frozen simplex.  This reduction, however, does not by itself produce a genuine semigroup increment; the main additional step is to prove that the frozen solution classes are compatible as the terminal time varies and hence assemble into a canonical quotient class for $\hat\delta$.  This yields canonical classes for $0<\gamma<1$, and at $\gamma=1$ under logarithmic control. We further provide a scale-dependent criterion for selecting genuine representatives, verified for heat semigroups on Sobolev scales through a parabolic Littlewood--Paley tail condition.
\end{abstract}

\makeatletter
\@namedef{subjclassname@2020}{\textup{2020} Mathematics Subject Classification}
\makeatother
\subjclass[2020]{60L20, 60H99, 60L90}

\keywords{Low-regularity Sewing; semigroup Sewing lemma; quotient spaces; analytic semigroups}

\maketitle

\tableofcontents

\setcounter{section}{0}

\allowdisplaybreaks

\section{Introduction}

\subsection{Motivation}
The Sewing lemma is a central analytic mechanism in rough path theory.  Starting from a local two-increment $A_{ts}$, it reconstructs an additive increment by controlling the three-point defect
\[
(\delta A)_{tus}=A_{ts}-A_{tu}-A_{us}.
\]
In the classical regime, the decisive assumption is that $\delta A$ has order strictly larger than $1$.  This is the analytic reason why Riemann sums converge and why the resulting correction is unique.  This mechanism underlies Young integration~\cite{Young36}, Lyons' rough path theory~\cite{Lyons98,LyonsQian02}, Gubinelli's controlled rough paths~\cite{Gub04}, the enriched-path formulation of Feyel and de La Pradelle~\cite{FePr06}, and the modern accounts \cite{FH20, FV10}.

The threshold $1$ is not merely technical.  Above it, the relevant regularity condition rules out non-trivial coboundary cocycles, which is the source of uniqueness.  Below it, such cocycles are no longer excluded by regularity alone. Broux and Zambotti showed that one can still carry out a low-regularity sewing construction in the regime $0<\gamma\le 1,$ but the resulting representative is neither unique nor canonical, and the endpoint $\gamma=1$ carries a logarithmic loss~\cite{BZ22}.  Thus the correct low-regularity philosophy is not the absence of sewing, but the loss of uniqueness at the level of actual representatives.

This paper studies the corresponding question in the semigroup setting.  Let $S=(S_t)_{t\ge0}$ be a semigroup.  For mild evolution equations, the relevant
increment relation is not ordinary additivity:
\[
 R_{ts}=R_{tu}+R_{us},
\]
but semigroup additivity:
\[
R_{ts}=R_{tu}+S_{t-u}R_{us}.
\]
Equivalently, one replaces the usual coboundary $\delta$ by the semigroup coboundary
\[
\hat\delta=\delta-a, \quad a_{ts}=S_{t-s}-\mathrm{Id}.
\]
This modified complex, together with the semigroup Sewing lemma in the regular regime $\gamma>1$, was introduced in the rough evolution equation framework of Gubinelli and Tindel~\cite{GT10}.  It is particularly well adapted to mild formulations of evolution equations and to convolution integrals.  Its analytic role is closely related to classical semigroup estimates for parabolic equations \cite{Henry81,Pazy83} and to stochastic convolution theory \cite{DPZ92}.

The main difficulty addressed here is the following. In the low-regularity regime $0<\gamma\le1$, one cannot expect a canonical semigroup increment $R$ solving
\[
\hat\delta R=H.
\]
Indeed, if $K$ is a semigroup cocycle, namely $\hat\delta K=0$, then $R+K$ solves the same equation. Thus the construction has to distinguish the existence of representatives, their non-uniqueness, and the canonical quotient class modulo semigroup cocycles.

\subsection{Main ideas and main results}
The first idea is to reduce the semigroup equation to ordinary low-regularity sewing on frozen simplices.  For a fixed terminal time $\tau$, the frozen transform $\Theta_\tau$ transports increments to the terminal time and satisfies the intertwining identity
\[
\delta(\Theta_\tau g)=\Theta_\tau(\hat\delta g).
\]
Thus a semigroup defect becomes an ordinary defect on the frozen simplex.  This allows us to use the Broux--Zambotti sewing mechanism at each frozen terminal time.

The second idea is that frozen representatives need not be strictly compatible when \(\tau\) varies.  What is compatible is not the representatives themselves, but their quotient classes.  On frozen simplices the defect is absorbed by ordinary cocycles, and after reassembly it corresponds to the ambiguity by semigroup cocycles.  This leads to compatible frozen quotient families and gives an unconditional quotient-level semigroup sewing construction.

The third idea is to go beyond the quotient object and construct genuine semigroup representatives on admissible domains.  This requires an additional selection mechanism.  Instead of imposing an artificial global splitting condition on the whole Banach space, we impose a scale-dependent splitting only on the residuals generated by the dyadic construction.  This distinction is essential for parabolic semigroups. Heat semigroups are smoothing and generally noninvertible, so the right mechanism is not global invertibility but low-frequency approximate inversion combined with high-frequency smoothing tails.  This is verified in the Sobolev heat setting by using Littlewood--Paley estimates~\cite{BCD11}.

The main results of the paper are organized according to the three-layer structure of low-regularity Sewing.

\begin{enumerate}
\item We develop the semigroup cochain complex associated with $\hat\delta$. The semigroup cochain complex is established through Lemmas~\ref{prop:hdelta-square-zero} and~\ref{prop:exactness}. The obstruction to uniqueness is isolated in Propositions~\ref{prop:quotient-target} and~\ref{prop:trivial-cocycles-above-one}, leading to the canonical target principle in Corollary~\ref{cor:canonical-target}.

\item We introduce the frozen terminal-time transform and prove that it reduces semigroup defects to ordinary low-regularity sewing problems on frozen simplices.  The key intertwining identity is Proposition~\ref{prop:frozen-intertwining}, and the equivalence between genuine semigroup increments and strictly compatible frozen increment families is Theorem~\ref{thm:frozen-equivalence}. The frozen dyadic construction is carried out in Proposition~\ref{thm:frozen-dyadic-sewing} and applied to semigroup defects in Theorem~\ref{prop:dyadic-frozen-defects}.

\item We construct the quotient-level semigroup sewing maps in the low-regularity regime.  The subcritical construction \(0<\gamma<1\) is Theorem~\ref{thm:quotient-semigroup-sewing-subcritical}, and the critical logarithmic construction at \(\gamma=1\) is Theorem~\ref{thm:critical-quotient-semigroup-sewing}.  These results produce canonical classes modulo semigroup cocycles, rather than canonical genuine increments.

\item We develop a mechanism for selecting genuine representatives from the canonical quotient class and use it to obtain concrete semigroup increments. The subcritical and critical selection results are Theorems~\ref{thm:subcritical-dyadic-sewing} and~\ref{thm:critical-dyadic-sewing}.  For heat semigroups on Sobolev scales, the admissibility of this selection mechanism is verified through Proposition~\ref{prop:heat-spectral-splitting-estimate}, which leads to the genuine heat-sewing result in Theorem~\ref{thm:heat-sewing-sobolev}.  With these representatives in hand, we construct normalized semigroup increments in Theorem~\ref{thm:normalized-convolution-primitive-revised} and derive the pathwise stochastic estimates in Theorem~\ref{thm:stochastic-sewing-estimates}.
\end{enumerate}

\raggedbottom
\subsection{Relation to existing Sewing lemmas}
The paper extends the low-regularity Broux--Zambotti viewpoint to the semigroup coboundary, but the extension is not formal. In the classical sewing case $\gamma>1$, the sewing correction is unique~\cite{Gub04}. In the low-regularity ordinary case $0<\gamma\le1$, Broux and Zambotti construct continuous representatives, but these representatives are not canonical~\cite{BZ22}. In the semigroup case, for $\gamma>1$, the semigroup Sewing lemma also yields a unique correction~\cite{GT10}.  For $0<\gamma\le 1$, the same non-uniqueness phenomenon appears with $\delta$ replaced by $\hat\delta$, and an additional compatibility problem has to be solved as the terminal time varies.  The representatives depend on choices, while the quotient class modulo semigroup cocycles is canonical. In the degenerate case \(S_t=\mathrm{Id}\), the modified coboundary reduces to the ordinary coboundary, and the construction recovers the normalized ordinary low-regularity sewing construction of Broux--Zambotti~\cite{BZ22}.

This point also explains the role of the frozen quotient construction.  The frozen construction is unconditional and canonical at the quotient level.  Passing from this class to an actual representative is a stronger problem and necessarily requires a selection mechanism. In parabolic examples, such a selection cannot be based on a global inverse of the semigroup.  It is instead obtained through frequency localization and smoothing tails, in a way consistent with the analytic structure of rough
evolution equations~\cite{GT10}.

\subsection{Organization of the paper}
The paper is organized as follows.  Section~\ref{sec:framework} introduces the semigroup cochain complex, the Banach-scale setting, and the quotient spaces
used throughout the paper.  Section~\ref{sec:frozen-reduction} develops the frozen terminal-time reduction, which relates the semigroup coboundary $\hat\delta $ to the ordinary coboundary \(\delta\) on frozen simplices. Section~\ref{sec:dyadic-frozen} carries out the ordinary low-regularity sewing construction on each frozen simplex.  Section~\ref{sec:quotient-semigroup-sewing} assembles the frozen quotient classes into the quotient-level semigroup sewing maps. Section~\ref{sec:representative-selection-sewing} develops the representative-selection mechanism and constructs genuine semigroup representatives on admissible domains.  Section~\ref{sec:heat-verification} verifies this mechanism for heat semigroups on Sobolev scales.  Finally, Sections~\ref{sec:convolution-primitives} and~\ref{sec:stochastic-consequences} discuss normalized semigroup increments and pathwise stochastic consequences.

\section{Algebraic and analytic framework}\label{sec:framework}
This section fixes the semigroup cochain complex and the Banach-scale setting for the low-regularity Sewing problem.  Below the threshold $1$, uniqueness fails
in genuine increment spaces, and the natural target is a quotient by semigroup cocycles.

Following the framework established in~\cite[Section 3.1]{GT10}, we fix a time $T>0$ and consider a Banach scale $(B_{\alpha})_{\alpha\in\mathbb{R}}$. We recall that a family of Banach spaces $(B_{\alpha})_{\alpha\in\mathbb{R}}$ is a Banach scale if for every $\alpha \le \beta$, the embedding $B_{\beta} \hookrightarrow B_{\alpha}$ is continuous. Let $S=(S_{t})_{t\ge0}$ be an analytic semigroup acting on this scale. Then $S$ satisfies the following standard smoothing and continuity estimates: for every $\alpha \in \mathbb{R}$ and $\varepsilon \in (0,1)$, there exists a constant $C_{\alpha,\varepsilon,T} > 0$ such that for all $t \in (0, T]$,
\[
\|S_t x\|_{B_{\alpha+\varepsilon}}
\le C_{\alpha,\varepsilon,T}\, t^{-\varepsilon}\|x\|_{B_\alpha},
\qquad
\|(S_t-\mathrm{Id})x\|_{B_\alpha}
\le C_{\alpha,\varepsilon,T}\, t^\varepsilon \|x\|_{B_{\alpha+\varepsilon}}.
\]

Having established the analytical properties of the operator semigroup, we now develop a corresponding algebraic framework to precisely characterize the mild increments arising in evolution equations.
For $n \ge 1$, we denote the $n$-dimensional simplex by
\[ \Delta_T^n := \{ (t_1, \dots, t_n) \in [0, T]^n : 0 \le t_n \le \dots \le t_1 \le T \}. \]
Let $ C(\Delta_T^n; B_\alpha)$ be the space of continuous maps from $\Delta_T^n$ to $B_\alpha$. We denote the standard coboundary operator by 
$\delta: C(\Delta_T^n; B_\alpha) \to C(\Delta_T^{n+1}; B_\alpha),$
defined for $g \in C(\Delta_T^n; B_\alpha)$ as
\begin{equation} \label{eq:delta-def}
(\delta g)_{t_1 \dots t_{n+1}} := \sum_{i=1}^{n+1} (-1)^i g_{t_1 \dots \hat{t}_i \dots t_{n+1}},
\end{equation}
where $\hat{t}_i$ means that this particular argument is omitted.

While the standard coboundary $\delta$ is suitable for characterizing ordinary increments, it does not account for the intrinsic evolution generated by the semigroup. To reconcile this algebraic structure with the mild formulation of evolution equations, we follow Gubinelli \& Tindel~\cite{GT10} and introduce a semigroup coboundary operator that incorporates the semigroup action.  Let 
\begin{equation}
\label{eq:a-def}
a_{ts}:=S_{t-s}-\mathrm{Id},
\qquad 0\le s\le t\le T.
\end{equation}
The operator $\hat{\delta}: C(\Delta_T^n; B_\alpha) \to C(\Delta_T^{n+1}; B_\alpha)$
is defined as
\begin{equation} \label{eq:delta_hat_def}
\hat{\delta} := \delta - a,
\end{equation}
where $a$ acts via left multiplication on the first two time variables
\[ (ag)_{t_1 \dots t_{n+1}} := a_{t_1 t_2} g_{t_2 \dots t_{n+1}}. \]
In particular, for $g \in C(\Delta_T; B_\alpha)$ and $h \in C(\Delta_T^2; B_\alpha)$, the operator $\hat{\delta}$ takes the following explicit forms for $0 \le s \le u \le t \le T$,
\begin{align}
 (\hat{\delta}g)_{ts} = g_t - S_{t-s}g_s,\quad  (\hat{\delta}h)_{tus} = h_{ts} - h_{tu} - S_{t-u}h_{us}. \label{eq:delta_hat_2}
\end{align}
We set
\[
Z\hat C(\Delta_T^n; B_\alpha):=\ker(\hat\delta)\cap C(\Delta_T^n; B_\alpha),
\qquad
\hat B(\Delta_T^n; B_\alpha):=\operatorname{Im}(\hat\delta)\cap C(\Delta_T^n; B_\alpha).
\]

With the semigroup coboundary $\hat{\delta}$ and the operator-valued increment $a$ formally defined, we now establish their fundamental algebraic properties. The following results demonstrate that this modified framework forms a well defined cochain complex, providing the necessary structural foundation for the subsequent Sewing lemma.
\begin{lemma}
\label{lem:delta-a}
The operator-valued increment $a$ from \eqref{eq:a-def} satisfies
\begin{equation*}
(\delta a)_{tus}=a_{tu}a_{us},
\qquad 0\le s\le u\le t\le T.
\end{equation*}
\end{lemma}

\begin{proof}
By the semigroup property,
$S_{t-s}=S_{t-u}S_{u-s}$.
Hence
\begin{align*}
(\delta a)_{tus}=&\ a_{ts}-a_{tu}-a_{us} \\
=&\ S_{t-s}-S_{t-u}-S_{u-s}+\mathrm{Id} \\
=&\ (S_{t-u}-\mathrm{Id})(S_{u-s}-\mathrm{Id}) \\
=&\ a_{tu}a_{us}.
\end{align*}
This completes the proof.
\end{proof}

\begin{lemma}
\label{prop:hdelta-square-zero}
For every $n\ge 1$, the operator $\hat\delta$ from \eqref{eq:delta_hat_def} satisfies
$\hat\delta^2=0$ on $C(\Delta_T^n; B_\alpha)$.
\end{lemma}

\begin{proof}
Let $g\in C(\Delta_T^n; B_\alpha)$. Since $\delta^2=0$, we have
\[
\hat\delta^2 g
=
(\delta-a)(\delta g-ag)
=
\delta^2g-\delta(ag)-a\delta g+aag
=
-\delta(ag)-a\delta g+aag.
\]
By a direct computation from \eqref{eq:delta-def},
$\delta(ag)=(\delta a)g-a\delta g$.
Therefore,
\[
\hat\delta^2 g
=
-(\delta a)g+a\delta g-a\delta g+aag
=
\bigl(aa-\delta a\bigr)g.
\]
Lemma \ref{lem:delta-a} yields $\delta a=aa$, hence $\hat\delta^2=0$.
\end{proof}

We now record the basic cochain property of the modified coboundary.

\begin{lemma}\label{prop:exactness}
For every $n\ge 1$,
\begin{equation}
\label{eq:exactness}
Z\hat C(\Delta_T^{n+1}; B_\alpha)=\hat B(\Delta_T^{n+1}; B_\alpha).
\end{equation}
\end{lemma}

\begin{proof}
The inclusion $\hat B(\Delta_T^{n+1}; B_\alpha)\subseteq Z\hat C(\Delta_T^{n+1}; B_\alpha)$ follows from
Lemma \ref{prop:hdelta-square-zero}. Let us prove the converse inclusion.
Take $h\in Z\hat C(\Delta_T^{n+1}; B_\alpha)$. Define
\[
f_{t_1\cdots t_n}:=(-1)^{n+1}h_{t_1\cdots t_n 0},
\quad  \forall(t_1,\ldots,t_n)\in \Delta_T^n.
\]
By \eqref{eq:delta-def} and \eqref{eq:delta_hat_def},
\begin{equation}
\label{eq:1}
(\hat\delta f)_{t_1\cdots t_{n+1}}
=
(-1)^{n+1}
\left(
\sum_{i=1}^{n+1}(-1)^i h_{t_1\cdots \widehat{t_i}\cdots t_{n+1}0}
-a_{t_1t_2}h_{t_2\cdots t_{n+1}0}
\right).
\end{equation}
Since $\hat\delta h=0$, evaluated at $(t_1,\ldots,t_{n+1},0)$, we have
\[
0
=
(\hat\delta h)_{t_1\cdots t_{n+1}0}
=
\sum_{i=1}^{n+1}(-1)^i h_{t_1\cdots \widehat{t_i}\cdots t_{n+1}0}+(-1)^{n+2} h_{t_1\cdots t_{n+1}}
-a_{t_1t_2}h_{t_2\cdots t_{n+1}0}.
\]
Separating the last term in the sum gives
\[
\sum_{i=1}^{n+1}(-1)^i h_{t_1\cdots \widehat{t_i}\cdots t_{n+1}0}
-a_{t_1t_2}h_{t_2\cdots t_{n+1}0}
=
(-1)^{n+1} h_{t_1\cdots t_{n+1}},
\]
and inserting this identity into \eqref{eq:1} yields
$(\hat\delta f)_{t_1\cdots t_{n+1}}=h_{t_1\cdots t_{n+1}}$.
Hence $h\in \hat B(\Delta_T^{n+1}; B_\alpha)$, proving \eqref{eq:exactness}.
\end{proof}

We measure the size of these increments by H\"older norms defined in the following way. Let $\mu>0$ and $\alpha\in\mathbb{R}$. For $g\in  C(\Delta_T^2; B_\alpha)$, define
\begin{equation}
\label{eq:c2-holder}
\|g\|_{\mu,\alpha}
:=
\sup_{0\le s<t\le T}
\frac{\|g_{ts}\|_{B_\alpha}}{|t-s|^\mu}.
\end{equation}
Denote by $C^\mu(\Delta_T^2; B_\alpha)$ the set of all $g\in C(\Delta_T^2; B_\alpha)$ such that
$\|g\|_{\mu,\alpha}<\infty$.
In the same way, for $h\in C(\Delta_T^3; B_\alpha)$ and $0<\rho<\mu$, define
\[
\|h\|_{\rho,\mu-\rho;\alpha}
:=
\sup_{0\le s<u<t\le T}
\frac{\|h_{tus}\|_{B_\alpha}}{|t-u|^\rho |u-s|^{\mu-\rho}}.
\]
Then set
\begin{equation}
\label{eq:c3-holder}
\|h\|_{\mu,\alpha}
:=
\inf
\left\{
\sum_{j=1}^m \|h_j\|_{\rho_j,\mu-\rho_j;\alpha}
\,\,\Big|\,\,
h=\sum_{j=1}^m h_j,\
m\in\mathbb{N},\
0<\rho_j<\mu
\right\}.
\end{equation}
Let $C^{\mu}(\Delta_T^3; B_\alpha)$ be the set of all $h\in C(\Delta_T^3; B_\alpha)$ such that
$\|h\|_{\mu,\alpha}<\infty$.
For $g\in C(\Delta_T^2; B_\alpha)$, define the critical logarithmic seminorm
\[
\|g\|_{1,\log;\alpha}
:=
\sup_{0\le s<t\le T}
\frac{\|g_{ts}\|_{B_\alpha}}{|t-s|\bigl(1+|\log|t-s|\,|\bigr)}.
\]
Let $C_{\log}^{1}(\Delta_T^2; B_\alpha)$ denote the set of all $g$ such that
$\|g\|_{1,\log;\alpha}<\infty$.
Now, we introduce a slight extension of the spaces defined above.  Let $\gamma\in(0,1]$ and $\alpha\in\mathbb{R}$.
If $\gamma\in(0,1)$, define
\begin{equation}
\label{eq:low-target}
E^{\gamma}(\Delta_T^2; B_\alpha)
:=
\bigcap_{0<\varepsilon<\gamma}C^{\gamma-\varepsilon}(\Delta_T^2; B_{\alpha+\varepsilon}).
\end{equation}
If $\gamma=1$, define
\begin{equation}
\label{eq:critical-target}
E_{\log}^{1}(\Delta_T^2; B_\alpha):=C_{\log}^{1}(\Delta_T^2; B_\alpha)\cap\bigcap_{0<\varepsilon<1}C^{1-\varepsilon}(\Delta_T^2; B_{\alpha+\varepsilon}).
\end{equation}
Define the cocycle classes
\[
\mathcal Z^{\gamma}(\Delta_T^2; B_\alpha)
:=
E^{\gamma}(\Delta_T^2; B_\alpha)\cap Z\hat C(\Delta_T^2; B_\alpha), \quad \mathcal Z_{\log}^{1}(\Delta_T^2; B_\alpha)
:=
E_{\log}^{1}(\Delta_T^2; B_\alpha)\cap Z\hat C(\Delta_T^2; B_\alpha),
\]
where $0<\gamma<1$. When an actual semigroup increment exists in the above regularity spaces, its canonical information is its class modulo semigroup cocycles.  We therefore define the corresponding quotient target spaces by
\begin{equation}
\label{eq:quotient-targets}
\mathfrak E_2^{\gamma}(\Delta_T^2; B_\alpha)
:=
E^{\gamma}(\Delta_T^2; B_\alpha)/\mathcal Z^{\gamma}(\Delta_T^2; B_\alpha),
\qquad
\mathfrak E_{\log}^{1}(\Delta_T^2; B_\alpha)
:=
E_{\log}^{1}(\Delta_T^2; B_\alpha)/\mathcal Z_{\log}^{1}(\Delta_T^2; B_\alpha).
\end{equation}
Finally, for later use, set
\[
Z\hat C^{\mu}(\Delta_T^3; B_\alpha)
:=
Z\hat C(\Delta_T^3; B_\alpha)\cap C^{\mu}(\Delta_T^3; B_\alpha).
\]

The preceding definitions only describe the non-uniqueness of actual semigroup increments, namely increments \(R\) satisfying \(\hat\delta R=h\).  Their existence requires a separate construction.  The construction in Sections~\ref{sec:quotient-semigroup} and~\ref{sec:critical-semigroup} will first produce compatible frozen quotient families.  Before that construction, we record the quotient-level uniqueness statements that hold whenever global semigroup increments are available.

\begin{proposition}\label{prop:quotient-target}
Let $\gamma\in(0,1)$, $\alpha\in\mathbb{R}$, and let
$h\in Z\hat C^{\gamma}(\Delta_T^3; B_\alpha)$. 
\begin{enumerate}
\item If $R,\widetilde R\in E^{\gamma}(\Delta_T^2; B_\alpha)$ satisfy $\hat\delta R=h=\hat\delta \widetilde R$, then
$R-\widetilde R\in \mathcal Z^{\gamma}(\Delta_T^2; B_\alpha)$.

\item For  $h\in Z\hat C^{1}(\Delta_T^3; B_\alpha)$, if $R,\widetilde R\in E_{\log}^{1}(\Delta_T^2; B_\alpha)$ satisfy $\hat\delta R=h=\hat\delta \widetilde R$, then $R-\widetilde R\in \mathcal Z_{\log}^{1}(\Delta_T^2; B_\alpha)$.
\end{enumerate}
\end{proposition}

\begin{proof}
We only prove the first conclusion, as the second one is identical.  By $\hat\delta R=h=\hat\delta \widetilde R$, we obtain  $\hat\delta(R-\widetilde R)=0$.
Since $E^{\gamma}(\Delta_T^2; B_\alpha)$ is a linear space, $R-\widetilde R\in E^{\gamma}(\Delta_T^2; B_\alpha)$,
and therefore
\[
R-\widetilde R\in  E^{\gamma}(\Delta_T^2; B_\alpha)\cap Z\hat C(\Delta_T^2; B_\alpha)=\mathcal Z^{\gamma}(\Delta_T^2; B_\alpha).\qedhere
\]
\end{proof}

We also record the corresponding fact in the regular case $\gamma>1$.

\begin{proposition}\label{prop:trivial-cocycles-above-one}
Let $\gamma>1$ and $\alpha\in\mathbb{R}$. Then
\[
Z\hat C(\Delta_T^2; B_\alpha)\cap C^{\gamma}(\Delta_T^2; B_\alpha)=\{0\}.
\]
\end{proposition}

\begin{proof}
Setting $h\in Z\hat C(\Delta_T^2; B_\alpha)\cap C^{\gamma}(\Delta_T^2; B_\alpha)$, since $h\in Z\hat C(\Delta_T^2; B_\alpha)$, relation \eqref{eq:delta_hat_2} gives
\begin{equation}
\label{eq:cocycle-identity}
h_{ts}=h_{tu}+S_{t-u}h_{us},
\qquad 0\le s\le u\le t\le T.
\end{equation}
Fix $0\le s<t\le T$ and a partition $\Pi=\{s=t_0<t_1<\cdots<t_m=t\}$. We then proceed by iterating \eqref{eq:cocycle-identity} over this partition.
First, we expand $h_{ts}$ (i.e., $h_{t_m t_0}$) at the point $t_{m-1}$, 
\[
h_{t_m t_0} = h_{t_m t_{m-1}} + S_{t_m - t_{m-1}} h_{t_{m-1} t_0}.
\]
Next, we expand the remaining term $h_{t_{m-1} t_0}$ at the point $t_{m-2}$,
\[
h_{t_{m-1} t_0} = h_{t_{m-1} t_{m-2}} + S_{t_{m-1} - t_{m-2}} h_{t_{m-2} t_0}.
\]
Substituting this expansion into the previous equation and using the semigroup property $S_aS_b=S_{a+b}$ gives
\begin{align*}
h_{t_m t_0} &= h_{t_m t_{m-1}} + S_{t_m - t_{m-1}} \left( h_{t_{m-1} t_{m-2}} + S_{t_{m-1} - t_{m-2}} h_{t_{m-2} t_0} \right) \\
&= h_{t_m t_{m-1}} + S_{t_m - t_{m-1}} h_{t_{m-1} t_{m-2}} + S_{t_m - t_{m-2}} h_{t_{m-2} t_0}.
\end{align*}
Notice that $S_{t-t_m} = S_0 = \text{Id}$. By iterating this process $m$ times for the partition $\Pi$, we eventually arrive at the following identity,
\begin{equation} \label{eq:partition-representation}
h_{ts} = \sum_{i=0}^{m-1} S_{t-t_{i+1}} h_{t_{i+1}t_i}.
\end{equation}
Let
\begin{equation}  \label{eq:semigroup-level-bound}
M_{\alpha,T}:=\sup_{0\le r\le T}\|S_r\|_{\mathcal L(B_\alpha, B_\alpha)}<\infty.
\end{equation}
Using \eqref{eq:c2-holder} and \eqref{eq:semigroup-level-bound}, we obtain
\[
\|h_{ts}\|_{B_\alpha}
\le
M_{\alpha,T}\sum_{i=0}^{m-1}\|h_{t_{i+1}t_i}\|_{B_\alpha}
\le
M_{\alpha,T}\|h\|_{\gamma,\alpha}\sum_{i=0}^{m-1}|t_{i+1}-t_i|^\gamma.
\]
Then
\[
\|h_{ts}\|_{B_\alpha}
\le
M_{\alpha,T}\|h\|_{\gamma,\alpha}\sum_{i=0}^{m-1}|t_{i+1}-t_i|\,|t_{i+1}-t_i|^{\gamma-1}
\le
M_{\alpha,T}\|h\|_{\gamma,\alpha}|\Pi|^{\gamma-1}\,(t-s).
\]
Since $\gamma>1$, letting $|\Pi|\to 0$ gives $\|h_{ts}\|_{B_\alpha}=0$. Thus $h_{ts}=0$
for all $s<t$, and continuity yields $h_{tt}=0$ as well. Hence $h=0$.
\end{proof}

The preceding propositions describe the ambiguity of  semigroup increments, without constructing a canonical one.  The next corollary records their canonical quotient information.

\begin{coro}\label{cor:canonical-target}
Let $\gamma>0$ and $\alpha\in\mathbb{R}$.

\begin{enumerate}
\item If $\gamma\in(0,1)$ and a semigroup increment $R\in E^{\gamma}(\Delta_T^2; B_\alpha)$ of a given semigroup defect is available, then its canonical information at this regularity is its class in $\mathfrak E_2^{\gamma}(\Delta_T^2; B_\alpha)$. \mlabel{it:canona}

\item If $\gamma=1$ and a semigroup increment $R\in E_{\log}^{1}(\Delta_T^2; B_\alpha)$ is available, then its canonical information is its class in $\mathfrak E_{\log}^{1}(\Delta_T^2; B_\alpha)$. \mlabel{it:canonb}

\item If $\gamma>1$, this quotient ambiguity disappears. \mlabel{it:canonc}
\end{enumerate}
\end{coro}

\begin{proof}
Items~(\mref{it:canona}) and~(\mref{it:canonb}) follow from Proposition~\ref{prop:quotient-target}. Indeed, any two semigroup increments of the same defect differ by an element of the corresponding semigroup cocycle space.  Hence only their quotient class is canonical.  Item~(\mref{it:canonc}) follows from Proposition~\ref{prop:trivial-cocycles-above-one}.
\end{proof}

\section{Frozen terminal-time reduction to the classical low-regularity Sewing problem}
\label{sec:frozen-reduction}
The purpose of this section is to reduce the semigroup Sewing equation 
\begin{equation}
\label{eq:semigroup-sewing-eq}
\hat\delta R = h
\end{equation}
to a family of ordinary Sewing equations on simplices with frozen terminal time. The key mechanism is a conjugation by the semigroup, which transforms the modified
coboundary $\hat\delta$ from Section~\ref{sec:framework} into the ordinary coboundary $\delta$. 

We first introduce the localized counterparts of the algebraic and analytical structures from Section~\ref{sec:framework}, restricting the time horizon to a frozen terminal point $\tau \in [0,T]$. For each terminal time $\tau \in [0,T]$, we define the frozen simplex 
$$\Delta_{\tau}^{n} := \{(t_{1}, \dots, t_{n}) \in [0, \tau]^{n} : 0 \le t_{n} \le \dots \le t_{1} \le \tau\}.$$
The ordinary increment spaces $C(\Delta_\tau^n;B_\beta)$ and their H\"older variants are defined as in Section~\ref{sec:framework}, but on $[0,\tau]$ and with the ordinary coboundary $\delta$ in place of $\hat\delta$. We also set
\[
ZC(\Delta_\tau^3; B_\beta):=\ker(\delta)\cap C(\Delta_\tau^3; B_\beta),
\quad
ZC^{\mu}(\Delta_\tau^3; B_\beta):=\ker(\delta)\cap C^{\mu}(\Delta_\tau^3; B_\beta).
\]
Correspondingly, the local target spaces are defined as
\[
F^{\gamma}(\Delta_\tau^2; B_\alpha) := \bigcap_{0 < \varepsilon < \gamma} C^{\gamma-\varepsilon}(\Delta_\tau^2; B_{\alpha+\varepsilon}), \quad F_{\log}^{1}(\Delta_\tau^2; B_\alpha):= C_{\log}^{1}(\Delta_\tau^2; B_\alpha) \cap \bigcap_{0 < \varepsilon < 1} C^{1-\varepsilon}(\Delta_\tau^2; B_{\alpha+\varepsilon}).
\]
With the local environments established, we introduce the following transformation which serves as the algebraic bridge between the semigroup and ordinary complexes.

\begin{definition}\label{def:frozen-transform}
Fix $\tau\in[0,T]$. For $n\ge 2$ and $g\in C(\Delta_T^n; B_\alpha)$, 
\[
\Theta_\tau: C(\Delta_T^n; B_\alpha)\to C(\Delta^n_\tau;B_\alpha)
\]
is called {\bf frozen terminal-time transform} if
\[
(\Theta_\tau g)_{t_1\cdots t_n}
=
S_{\tau-t_1}g_{t_1\cdots t_n},
\quad
\forall \,(t_1,\ldots,t_n)\in\Delta_\tau^n.
\]
\end{definition}

The following results establish how this transform intertwines the coboundaries and preserves regularity, effectively reducing semigroup defects to ordinary Sewing data.

\begin{proposition}\label{prop:frozen-intertwining}
Let $\tau\in[0,T]$ and $\Theta_\tau$ be a frozen terminal-time transform. Then for $n=2, 3$ and every $g\in C(\Delta_T^n; B_\alpha)$,
\[\delta(\Theta_\tau g)=\Theta_\tau(\hat\delta g).\]
\end{proposition}

\begin{proof}
Take $g\in C(\Delta_T^2; B_\alpha)$ and $(t,u,s)\in \Delta_\tau^3$. By
\eqref{eq:delta-def}, \eqref{eq:delta_hat_def}, and the semigroup property,
\begin{align*}
\delta(\Theta_\tau g)_{tus}
&=
(\Theta_\tau g)_{ts}
-(\Theta_\tau g)_{tu}
-(\Theta_\tau g)_{us}\\
&=
S_{\tau-t}g_{ts}
-S_{\tau-t}g_{tu}
-S_{\tau-u}g_{us} \\
&=
S_{\tau-t}\bigl(g_{ts}-g_{tu}-S_{t-u}g_{us}\bigr) \\
&=
\Theta_\tau(\hat\delta g)_{tus},
\end{align*}
which proves $\delta(\Theta_\tau g)=\Theta_\tau(\hat\delta g)$. The case $n=3$ is similar to this.
\end{proof}

\begin{coro}
\label{cor:frozen-defects}
Let $h\in Z\hat C(\Delta_T^3;B_\alpha)$ and define
\begin{equation}
\label{eq:frozen-defect}
h^\tau := \Theta_\tau h, \quad  \forall\tau\in[0,T].
\end{equation}
Then $h^\tau \in ZC_3(\Delta_\tau;B_\alpha)$. Moreover, for every $\tau\in[0,T]$ there exists $b^\tau\in C(\Delta_\tau^2; B_\alpha)$ such that $\delta b^\tau = h^\tau$.
\end{coro}

\begin{proof}
Since $h\in Z\hat C(\Delta_T^3;B_\alpha)$, Proposition~\ref{prop:frozen-intertwining} gives
\[
\delta h^\tau=\delta(\Theta_\tau h)=\Theta_\tau(\hat\delta h)=0,
\]
which proves $h^\tau \in ZC_3(\Delta_\tau;B_\alpha)$.
Furthermore,  for $ 0\le s\le u\le t\le \tau$ define $b^\tau_{ts}:=-h^\tau_{ts0}$  and
\[
(\delta b^\tau)_{tus}=-b^\tau_{us}+b^\tau_{ts}-b^\tau_{tu}=h^\tau_{us0}-h^\tau_{ts0}+h^\tau_{tu0}.
\]
On the other hand, since $\delta h^\tau=0$,
\[
0=(\delta h^\tau)_{tus0}=-h^\tau_{us0}+h^\tau_{ts0}-h^\tau_{tu0}+h^\tau_{tus}.
\]
Therefore
\[
h^\tau_{us0}-h^\tau_{ts0}+h^\tau_{tu0}=h^\tau_{tus},
\]
which is exactly $\delta b^\tau = h^\tau$.
\end{proof}

After the algebraic reduction, we turn to the regularity bounds for the frozen transform.

\begin{proposition}
\label{prop:frozen-regularity}
Set $\gamma\in(0,1]$, $\alpha\in\mathbb{R}$, and $\tau\in[0,T]$. Let  $h\in Z\hat C^{\gamma}(\Delta_T^3;B_\alpha)$ and define $h^\tau$ by \eqref{eq:frozen-defect}. Then
\begin{equation}\label{eq:frozen-defect-regularity}
h^\tau \in ZC^{\gamma}(\Delta_\tau^3;B_\alpha) \quad \text{and} \quad \|h^\tau\|_{\gamma,\alpha;\tau}
\le
M_{\alpha,T}\|h\|_{\gamma,\alpha}.
\end{equation}
Furthermore, if $\gamma\in(0,1)$ and $R\in E^{\gamma}(\Delta_T^2; B_\alpha)$, then $\Theta_\tau R \in F^{\gamma}(\Delta_\tau^2; B_\alpha)$
and for every $0<\varepsilon<\gamma$,
\[\|\Theta_\tau R\|_{\gamma-\varepsilon,\alpha+\varepsilon;\tau}\le M_{\alpha+\varepsilon,T}\|R\|_{\gamma-\varepsilon,\alpha+\varepsilon};\]
if $R\in E_{\log}^{1}(\Delta_T^2; B_\alpha)$, then $\Theta_\tau R \in F_{\log}^{1}(\Delta_\tau^2; B_\alpha)$ and for every $\varepsilon\in(0,1)$,
\[\|\Theta_\tau R\|_{1-\varepsilon,\alpha+\varepsilon;\tau}\le M_{\alpha+\varepsilon,T}\|R\|_{1-\varepsilon,\alpha+\varepsilon},\]
while $\|\Theta_\tau R\|_{1,\log;\tau,\alpha}\le M_{\alpha,T} \|R\|_{1,\log;\alpha}$.
\end{proposition}

\begin{proof}
Since $h^\tau \in ZC_3(\Delta_\tau;B_\alpha)$ by Corollary~\ref{cor:frozen-defects}, it remains only to estimate its norm. Let $h=\sum_{j=1}^m h_j$ be an admissible decomposition in \eqref{eq:c3-holder}, with
parameters $0<\rho_j<\gamma$. Then
\[
\|(\Theta_\tau h_j)_{tus}\|_{B_\alpha}=\|S_{\tau-t}h_{j,tus}\|_{B_\alpha}\overset{\eqref{eq:semigroup-level-bound}}{\le}M_{\alpha,T}\|h_{j,tus}\|_{B_\alpha}.
\]
Hence
\[
\|\Theta_\tau h_j\|_{\rho_j,\gamma-\rho_j;\alpha}\le M_{\alpha,T}\|h_j\|_{\rho_j,\gamma-\rho_j;\alpha}.
\]
Taking the infimum over all admissible decompositions yields \eqref{eq:frozen-defect-regularity}. The remaining conclusions are proved in the same way.
\end{proof}

To ensure coherence across varying terminal times, we introduce the following notion of compatibility. For later use, set
\[
\mathfrak D_T^3:=\{(\tau,t,s)\in[0,T]^3:\ 0\le s\le t\le \tau\le T\}.
\]

\begin{definition}\label{def:compatible-frozen-family}
Fix $h\in Z\hat C(\Delta_T^3;B_\alpha)$ and set $h^\tau:=\Theta_\tau h$. A compatible frozen increment family over $h$ is a family $r=(r^\tau)_{\tau\in[0,T]}$
such that
\begin{enumerate}
\item for every $\tau\in[0,T]$, $r^\tau\in C(\Delta_\tau^2;B_\alpha)$;

\item the map $ \mathfrak D_T^3\to B_\alpha, \, (\tau,t,s)\mapsto r^\tau_{ts}$ is continuous;

\item for every $\tau\in[0,T]$, $\delta r^\tau = h^\tau$;

\item for all $0\le s\le t\le \tau'\le \tau\le T$, $r^\tau_{ts}=S_{\tau-\tau'}r^{\tau'}_{ts}$.
\end{enumerate}
\end{definition}

The following results establish the bijective correspondence between semigroup increments and compatible frozen families, thereby justifying the reduction of the Sewing problem.

\begin{proposition}\label{prop:semigroup-to-frozen}
Let $h\in Z\hat C(\Delta_T^3;B_\alpha)$, $R\in  C(\Delta_T^2;B_\alpha)$ and $\hat\delta R = h$. For $\tau\in[0,T]$, define $r^\tau := \Theta_\tau R$.
Then $(r^\tau)_{\tau\in[0,T]}$ is a compatible frozen increment family over $h$.
Moreover, if $\gamma\in(0,1)$ and $R\in E^{\gamma}(\Delta_T^2; B_\alpha)$, then for every
$0<\varepsilon<\gamma$,
\[
\sup_{\tau\in[0,T]}\|r^\tau\|_{\gamma-\varepsilon,\alpha+\varepsilon;\tau}\le M_{\alpha+\varepsilon,T}\|R\|_{\gamma-\varepsilon,\alpha+\varepsilon};
\]
if $R\in E_{\log}^{1}(\Delta_T^2; B_\alpha)$, then for every $\varepsilon\in(0,1)$,
\[
\sup_{\tau\in[0,T]}\|r^\tau\|_{1-\varepsilon,\alpha+\varepsilon;\tau}\le M_{\alpha+\varepsilon,T} \|R\|_{1-\varepsilon,\alpha+\varepsilon},
\]
and $ \sup_{\tau\in[0,T]}\|r^\tau\|_{1,\log;\tau,\alpha}\le M_{\alpha,T}\|R\|_{1,\log;\alpha}$.
\end{proposition}

\begin{proof}
By Definition~\ref{def:frozen-transform}, each $r^\tau$ belongs to $C(\Delta_\tau^2;B_\alpha)$. The continuity of Definition~\ref{def:compatible-frozen-family} (b) follows from the continuity of $R$
and the strong continuity of the semigroup. Since $\hat\delta R=h$, Proposition~\ref{prop:frozen-intertwining} yields
\[
\delta r^\tau=\delta(\Theta_\tau R)=\Theta_\tau(\hat\delta R)=\Theta_\tau h=h^\tau,
\]
which proves Definition~\ref{def:compatible-frozen-family} (c). Next, let $0\le s\le t\le \tau'\le \tau\le T$. Then
\[
r^\tau_{ts}=(\Theta_\tau R)_{ts}=S_{\tau-t}R_{ts}=S_{\tau-\tau'}S_{\tau'-t}R_{ts}=S_{\tau-\tau'}(\Theta_{\tau'} R)_{ts}=S_{\tau-\tau'}r^{\tau'}_{ts},
\]
which is exactly Definition~\ref{def:compatible-frozen-family} (d).
The uniform regularity bounds follow directly from Proposition~\ref{prop:frozen-regularity}.
\end{proof}

\begin{proposition}\label{prop:frozen-to-semigroup}
Let $h\in Z\hat C(\Delta_T^3;B_\alpha)$ and let $(r^\tau)_{\tau\in[0,T]}$ be a compatible frozen increment family over $h$. For $0\le s\le t\le T$, define $R_{ts}:=r^t_{ts}$. Then
\[
R\in C(\Delta_T^2;B_\alpha)\  \text{\,and\,} \ \hat\delta R = h.
\]
Moreover, if $\gamma\in(0,1)$ and for every $0<\varepsilon<\gamma$, $$ \sup_{\tau\in[0,T]}\|r^\tau\|_{\gamma-\varepsilon,\alpha+\varepsilon;\tau}<\infty,$$ then
$R\in E^{\gamma}(\Delta_T^2; B_\alpha)$ and for every $0<\varepsilon<\gamma$,
\begin{equation}
\label{eq:reconstructed-subcritical-estimate}
\|R\|_{\gamma-\varepsilon,\alpha+\varepsilon}
\le
\sup_{\tau\in[0,T]}
\|r^\tau\|_{\gamma-\varepsilon,\alpha+\varepsilon;\tau}.
\end{equation}
Similarly, if 
$$\sup_{\tau\in[0,T]}\|r^\tau\|_{1,\log;\tau,\alpha}+\sup_{\tau\in[0,T]}\|r^\tau\|_{1-\varepsilon,\alpha+\varepsilon;\tau}<\infty$$
for every $\varepsilon\in(0,1)$, then $R\in E_{\log}^{1}(\Delta_T^2; B_\alpha)$ and
\[
\|R\|_{1,\log;\alpha}\le\sup_{\tau\in[0,T]}\|r^\tau\|_{1,\log;\tau,\alpha}, \quad \|R\|_{1-\varepsilon,\alpha+\varepsilon}\le\sup_{\tau\in[0,T]}\|r^\tau\|_{1-\varepsilon,\alpha+\varepsilon;\tau}.
\]
\end{proposition}

\begin{proof}
The continuity of $R$ on $\Delta_T^2$ follows from Definition~\ref{def:compatible-frozen-family} (b). Let $(t,u,s)\in\Delta_T^3$. Then
\begin{align*}
(\hat\delta R)_{tus}&=\ R_{ts}-R_{tu}-S_{t-u}R_{us} \\
&=\ r^t_{ts}-r^t_{tu}-S_{t-u}r^u_{us} \hspace{2cm}(\text{by $R_{ts}:=r^t_{ts}$})\\
&=\ r^t_{ts}-r^t_{tu}-r^t_{us} \hspace{2cm}(\text{by Definition~\ref{def:compatible-frozen-family} (d)})\\
&=\ (\delta r^t)_{tus} \\
&=\ h^t_{tus} \hspace{2cm}(\text{by Definition~\ref{def:compatible-frozen-family} (c)}).
\end{align*}
Since $h^t_{tus}=S_{t-t}h_{tus}=h_{tus}$, we obtain $\hat\delta R=h$.
Now suppose $$\sup_{\tau\in[0,T]}\|r^\tau\|_{\gamma-\varepsilon,\alpha+\varepsilon;\tau}<\infty.$$ 
Then for
$0\le s<t\le T$,
\[
\|R_{ts}\|_{B_{\alpha+\varepsilon}}
=
\|r^t_{ts}\|_{B_{\alpha+\varepsilon}}
\le
\Bigl(
\sup_{\tau\in[0,T]}
\|r^\tau\|_{\gamma-\varepsilon,\alpha+\varepsilon;\tau}
\Bigr)
|t-s|^{\gamma-\varepsilon}.
\]
Taking the supremum over $s<t$ proves \eqref{eq:reconstructed-subcritical-estimate}, hence $R\in E^{\gamma}(\Delta_T^2; B_\alpha)$.
The critical case is proved in the same way, by evaluating the corresponding frozen bounds at $\tau=t$.
\end{proof}

Combining the two directions gives the equivalence between semigroup increments and compatible frozen increment families.

\begin{theorem}\label{thm:frozen-equivalence}
Let $h\in Z\hat C(\Delta_T^3;B_\alpha)$. Then the map
\[
\mathcal T_h:\{R\in C(\Delta_T^2;B_\alpha) \mid \hat\delta R=h\}\to \{\text{compatible frozen increment families over }h\}
\]
defined by $\mathcal T_h(R):=(\Theta_\tau R)_{\tau\in[0,T]}$ is a bijection. Its inverse is given by
\begin{equation}
\label{eq:frozen-equivalence-inverse}
(r^\tau)_{\tau\in[0,T]}\mapsto R, \quad R_{ts}:=r^t_{ts}.
\end{equation}
In addition, if $\gamma\in(0,1)$ and $h\in Z\hat C^{\gamma}(\Delta_T^3;B_\alpha)$, then
$\mathcal T_h$ restricts to a bijection between $ \{R\in E^{\gamma}(\Delta_T^2; B_\alpha) \mid \hat\delta R=h\}$ and the compatible frozen increment families over $h$ satisfying
\begin{equation}
\label{eq:frozen-equivalence-subcritical-target}
\sup_{\tau\in[0,T]}\|r^\tau\|_{\gamma-\varepsilon,\alpha+\varepsilon;\tau}<\infty, \quad \forall \varepsilon\in(0, \gamma).
\end{equation}
Likewise, if $h\in Z\hat C^{1}(\Delta_T^3;B_\alpha)$, then $\mathcal T_h$ restricts to a bijection between $\{R\in E_{\log}^{1}(\Delta_T^2; B_\alpha) \mid \hat\delta R=h\}$ and the compatible frozen increment families satisfying
\begin{equation}
\label{eq:frozen-equivalence-critical-target}
\sup_{\tau\in[0,T]}\|r^\tau\|_{1,\log;\tau,\alpha}+\sup_{\tau\in[0,T]}\|r^\tau\|_{1-\varepsilon,\alpha+\varepsilon;\tau}<\infty, \quad \forall \varepsilon\in(0,1).
\end{equation}
\end{theorem}

\begin{proof}
Proposition~\ref{prop:semigroup-to-frozen} shows that $\mathcal T_h$ is well defined. Proposition~\ref{prop:frozen-to-semigroup} shows that \eqref{eq:frozen-equivalence-inverse} is also well defined. It remains to check that they are inverse to each other. If $R\in C(\Delta_T^2;B_\alpha)$ with $\hat\delta R=h$, then
\[
(\mathcal T_h(R))^t_{ts}=(\Theta_t R)_{ts}=S_{t-t}R_{ts}=R_{ts},
\]
so the inverse map sends $\mathcal T_h(R)$ back to $R$. Conversely, let $(r^\tau)_{\tau\in[0,T]}$ be a compatible frozen increment family. Then, for $0\le s\le t\le \tau\le T$,
\[
(\Theta_\tau R)_{ts}=S_{\tau-t}R_{ts}=S_{\tau-t}r^t_{ts}=r^\tau_{ts},
\]
where the last equality follows from Definition~\ref{def:compatible-frozen-family} (d). Hence $\mathcal T_h(R)=(r^\tau)_{\tau\in[0,T]}$.
The low-regularity statements are immediate from the uniform estimates proved in
Propositions~\ref{prop:semigroup-to-frozen} and \ref{prop:frozen-to-semigroup}.
\end{proof}

Based on the equivalence established above, we arrive at the following reduction principle, which provides a clear operational path for our construction.

\begin{coro}\label{cor:reduction-principle}
Let $\gamma\in(0,1]$, $\alpha\in\mathbb{R}$, and $h\in Z\hat C^{\gamma}(\Delta_T^3; B_\alpha)$.
\begin{enumerate}
\item For every $\tau\in[0,T]$, the frozen defect $h^\tau$ defined in \eqref{eq:frozen-defect} satisfies $h^\tau\in ZC^{\gamma}(\Delta_\tau^3;B_\alpha)$. \label{it:reductia}

\item In order to construct a semigroup increment $R$ solving \eqref{eq:semigroup-sewing-eq}, it is enough to construct a compatible frozen increment family
$(r^\tau)_{\tau\in[0,T]}$ over $h$ such that the bounds in \eqref{eq:frozen-equivalence-subcritical-target} hold when $\gamma\in(0,1)$, or the bounds in \eqref{eq:frozen-equivalence-critical-target} hold when $\gamma=1$. \label{it:reductib}

\item Once such a family is available, the formula $ R_{ts}=r^t_{ts}$ defines a semigroup increment with
\[
R\in
\begin{cases}
E^{\gamma}(\Delta_T^2; B_\alpha), & 0<\gamma<1,\\[0.3em]
E_{\log}^{1}(\Delta_T^2; B_\alpha), & \gamma=1,
\end{cases}
\ \text{\, and\, }\ \hat\delta R=h.
\] \label{it:reductic}
\end{enumerate}
\end{coro}

\begin{proof}
Item~(\ref{it:reductia}) is Corollary~\ref{cor:frozen-defects} together with Proposition~\ref{prop:frozen-regularity}. Items~(\ref{it:reductib}) and~(\ref{it:reductic}) are direct consequences of Theorem~\ref{thm:frozen-equivalence}.
\end{proof}

\begin{remark}
When \(S_t=\mathrm{Id}\), one has \(a=0\) and hence \(\hat\delta=\delta\).  Moreover, the frozen transform \(\Theta_\tau\) is simply the restriction to \(\Delta_\tau\).  Thus the preceding reduction becomes the ordinary low-regularity Sewing problem on \([0,\tau]\).  With the same dyadic normalization, the construction reduces to the normalized ordinary sewing construction of Broux--Zambotti~\cite{BZ22}.
\end{remark}

\section{Localized ordinary Sewing under dyadic normalization}\label{sec:dyadic-frozen}
After the terminal time $\tau$ is frozen, the semigroup equation $\hat\delta R=h$ becomes an ordinary Sewing problem on the simplex $\Delta_\tau$.  In this section we fix a dyadic ordinary Sewing procedure for such frozen problems in the regime $\gamma\in(0,1]$.

We begin by fixing the dyadic meshes on which the ordinary low-regularity Sewing construction will be normalized. For $\tau\in[0,T]$ and $n\in\mathbb N$, we set
\[
\mathcal D_n^\tau:=\{k2^{-n}\tau \mid k=0,\ldots,2^n\}, \quad \mathcal D^\tau:=\bigcup_{n\ge 0}\mathcal D_n^\tau.
\]
The dyadic grids above are not yet an analytical object; rather, they identify the normalization that will be used throughout the section. Once the normalization is
fixed, the ordinary Sewing map becomes linear, and this linearity will be essential when we return to semigroup defects. 
Let $\tau\in[0,T]$, $\beta\in\mathbb R$, and $\gamma\in(0,1]$. We set
\[
\mathcal A^{\gamma}(\Delta_\tau^2;B_\beta):=\{A\in C(\Delta_\tau^2;B_\beta) \mid \delta A\in C^{\gamma}(\Delta_\tau^3;B_\beta)\}.
\]
The previous definition only records the natural domain of the ordinary Sewing problem. The next statement is the corresponding localized dyadic Sewing theorem
on the interval $[0,\tau]$.

\begin{lemma}\label{prop:localized-ordinary-dyadic}
Let $\tau\in[0,T]$, $\beta\in\mathbb R$, and $\gamma\in(0,1]$. Then there exists a linear map
\[
\mathcal I_\tau^{\mathrm d}: \mathcal A^{\gamma}(\Delta_\tau^2;B_\beta) \to C([0,\tau];B_\beta),
\]
obtained from a dyadic recursion on $\mathcal D^\tau$ and continuous extension to $[0,\tau]$, such that, for every $A\in \mathcal A^{\gamma}(\Delta_\tau^2;B_\beta)$, $\mathcal I_\tau^{\mathrm d}(A)_0=0$. If we define the dyadic remainder $\mathcal R_\tau^{\mathrm d}(A):=A-\delta\mathcal I_\tau^{\mathrm d}(A)$, then
$
\delta \mathcal R_\tau^{\mathrm d}(A)=\delta A.
$
Moreover, if $\gamma\in(0,1)$, then
\begin{equation}
\label{eq:localized-dyadic-subcritical-estimate}
\mathcal R_\tau^{\mathrm d}(A)\in C^{\gamma}(\Delta_\tau^2;B_\beta)\ \text{\, and \, }\ \|\mathcal R_\tau^{\mathrm d}(A)\|_{\gamma,\beta;\tau}\le C_\gamma\|\delta A\|_{\gamma,\beta;\tau};
\end{equation}
if $\gamma=1$, then
\begin{equation}
\label{eq:localized-dyadic-critical-estimate}
\mathcal R_\tau^{\mathrm d}(A)\in C_{\log}^{1}(\Delta_\tau^2;B_\beta)\ \text{\, and \, }\ \|\mathcal R_\tau^{\mathrm d}(A)\|_{1,\log;\tau,\beta}\le C_{T}\|\delta A\|_{1,\beta;\tau}.
\end{equation}
while, for every $\varepsilon\in(0,1)$, $\|\mathcal R_\tau^{\mathrm d}(A)\|_{1-\varepsilon,\beta;\tau}\le C_{\varepsilon,T}\|\delta A\|_{1,\beta;\tau}$. Here the constants can be chosen independently of $\tau\in[0,T]$.
\end{lemma}

\begin{proof}
The existence of the dyadic-normalized correction $\mathcal R_\tau^d(A)$ satisfying the required estimates follows directly from the normalized ordinary low-regularity Sewing lemma; see \cite[Theorem~3.2]{BZ22}.  Moreover, the construction preserves the defect in the sense that
\[
\delta \mathcal R_\tau^{d}(A)=\delta A,
\]
which follows from the normalization construction and the identity \(\delta^2=0\).

The constants in the subcritical estimate depend only on \(\gamma\).  In the critical case the logarithmic bound is taken uniformly for \(\tau\le T\), which gives the constant \(C_{T}\), and the same bound implies the stated \(C^{1-\varepsilon}\) estimate.  The argument of \cite[Theorem~3.2]{BZ22} is unchanged for \(B_\beta\)-valued increments, since the dyadic recursion uses only linearity and the triangle inequality in the target Banach space.
\end{proof}

From now on, the notation $\mathcal I_\tau^{\mathrm d}$ and $\mathcal R_\tau^{\mathrm d}$ will refer to one fixed choice of the dyadic operators provided by Lemma~\ref{prop:localized-ordinary-dyadic}. This is the first place where a normalization enters the argument. To apply Lemma~\ref{prop:localized-ordinary-dyadic} to semigroup defects, we first freeze the semigroup defect and choose a two increment whose ordinary coboundary equals the resulting ordinary three cocycle. We then apply the dyadic Sewing operator to this chosen two-increment.  Let $\tau\in[0,T]$, $\beta\in\mathbb R$, and let $h\in ZC(\Delta^3_\tau;B_\beta)$. We define
\begin{equation}
\label{eq:canonical-frozen-primitive}
b^h_{ts}:=-h_{ts0}, \quad 0\le s\le t\le \tau.
\end{equation}
The formula above gives a two-increment whose ordinary coboundary is $h$. We recall it here because this two increment is used as the input to the dyadic Sewing procedure.

\begin{lemma}\label{lem:canonical-frozen-primitive}
Let $\tau\in[0,T]$, $\beta\in\mathbb R$, and $h\in ZC(\Delta^3_\tau;B_\beta)$. Then $\delta b^h = h$. Consequently, if $h\in ZC^{\gamma}(\Delta^3_\tau;B_\beta)$ for some $\gamma\in(0,1]$, then $b^h\in \mathcal A^{\gamma}(\Delta_\tau^2;B_\beta)$.
\end{lemma}

\begin{proof}
Since $h\in ZC(\Delta^3_\tau;B_\beta)$, one has
\[
0=(\delta h)_{tus0}=-h_{us0}+h_{ts0}-h_{tu0}+h_{tus}.
\]
Rearranging this identity gives
\[
h_{us0}-h_{ts0}+h_{tu0}=h_{tus},
\]
which is exactly $\delta b^h = h$.
Assume that $h\in ZC^{\gamma}(\Delta^3_\tau;B_\beta)$. By the definition of $A^{\gamma}(\Delta_\tau^2;B_\beta)$ and $\delta b^h = h$, we have $b^h\in \mathcal A^{\gamma}(\Delta_\tau^2;B_\beta)$.
\end{proof}

We next combine the two increment $b^h$ with the localized dyadic Sewing operator. Let $\tau\in[0,T]$, $\beta\in\mathbb R$ and $\gamma\in(0,1]$. For
$h\in ZC^{\gamma}(\Delta^3_\tau;B_\beta)$, we define
\[
\Lambda_\tau^{\mathrm d}h:=\mathcal R_\tau^{\mathrm d}(b^h)=b^h-\delta\mathcal I_\tau^{\mathrm d}(b^h).
\]

\begin{proposition}\label{thm:frozen-dyadic-sewing}
Let $\tau\in[0,T]$ and $\beta\in\mathbb R$. If $\gamma\in(0,1)$, then
\[
\Lambda_\tau^{\mathrm d}: ZC^{\gamma}(\Delta^3_\tau;B_\beta)\to C^{\gamma}(\Delta^2_\tau;B_\beta)
\]
is linear, satisfies $\delta\Lambda_\tau^{\mathrm d}h = h$ and obeys the estimate $\|\Lambda_\tau^{\mathrm d}h\|_{\gamma,\beta;\tau}\le C_\gamma\|h\|_{\gamma,\beta;\tau}$. 
If $\gamma=1$, then
\[
\Lambda_\tau^{\mathrm d}: ZC^{1}(\Delta^3_\tau;B_\beta)\to C^{1}_{\mathrm{log}}(\Delta^2_\tau;B_\beta)
\]
is linear, satisfies $\delta\Lambda_\tau^{\mathrm d}h = h$ and obeys the estimates $\|\Lambda_\tau^{\mathrm d}h\|_{1,\log;\tau,\beta} \le C_{1,T}\|h\|_{1,\beta;\tau}$, as well as, for every $\varepsilon\in(0,1)$,  $\|\Lambda_\tau^{\mathrm d}h\|_{1-\varepsilon,\beta;\tau}\le C_{\varepsilon,T}\|h\|_{1,\beta;\tau}$.
\end{proposition}

\begin{proof}
The map $h\mapsto b^h$ is linear by definition, and the map $A\mapsto \mathcal R_\tau^{\mathrm d}(A)$ is linear by Lemma~\ref{prop:localized-ordinary-dyadic}. Therefore $\Lambda_\tau^{\mathrm d}$ is linear as well.
Next, let $\gamma\in(0,1]$ and $h\in ZC^{\gamma}(\Delta^3_\tau;B_\beta)$. By Lemma~\ref{lem:canonical-frozen-primitive}, $b^h\in\mathcal A^{\gamma}(\Delta_\tau^2;B_\beta)$ and $\delta b^h = h$. Applying Lemma~\ref{prop:localized-ordinary-dyadic} with $A=b^h$, we obtain
\[
\delta\Lambda_\tau^{\mathrm d}h=\delta\mathcal R_\tau^{\mathrm d}(b^h)=\delta b^h=h.
\]
The estimates follows immediately from \eqref{eq:localized-dyadic-subcritical-estimate}, \eqref{eq:localized-dyadic-critical-estimate} and the identity $\delta b^h=h$.
\end{proof}

We now return to the semigroup setting. By Corollary~\ref{cor:reduction-principle}, a semigroup defect $h\in Z\hat C^{\gamma}(\Delta^3_T;B_\alpha)$ gives rise, for every terminal time $\tau$, to an ordinary cocycle $h^\tau := \Theta_\tau h$. Proposition~\ref{thm:frozen-dyadic-sewing} can therefore be applied simplex by simplex.

\begin{definition}\label{def:dyadic-frozen-defect-family}
Let $\gamma\in(0,1]$, $\alpha\in\mathbb R$, and $h\in Z\hat C^{\gamma}(\Delta^3_T;B_\alpha)$. For each $\tau\in[0,T]$, set $h^\tau:=\Theta_\tau h$ and define $r_{\mathrm d}^\tau(h):=\Lambda_\tau^{\mathrm d}(h^\tau)$. The family $r_{\mathrm d}(h):=\bigl(r_{\mathrm d}^\tau(h)\bigr)_{\tau\in[0,T]}$ is called the {\bf dyadic frozen increment family} associated with $h$.
\end{definition}

The next result records the simplexwise defect identity and regularity of$r_{\mathrm d}(h)$, without asserting terminal compatibility.

\begin{theorem}\label{prop:dyadic-frozen-defects}
Let $\gamma\in(0,1]$, $\alpha\in\mathbb R$ and $h\in Z\hat C^{\gamma}(\Delta^3_T;B_\alpha)$.  Then, for every $\tau\in[0,T]$, if $\gamma\in(0,1)$, one has
\[
r_{\mathrm d}^\tau(h)\in C^{\gamma}(\Delta^2_\tau;B_\alpha), \quad \delta r_{\mathrm d}^\tau(h)=h^\tau
\]
and the uniform estimate $\sup_{\tau\in[0,T]}\|r_{\mathrm d}^\tau(h)\|_{\gamma,\alpha;\tau}\le C_\gamma M_{\alpha,T}\|h\|_{\gamma,\alpha}$; if $\gamma=1$, one has
\[
r_{\mathrm d}^\tau(h)\in C^{1}_{\mathrm{log}}(\Delta^2_\tau;B_\alpha), \quad \delta r_{\mathrm d}^\tau(h)=h^\tau
\]
and the uniform estimates $$\sup_{\tau\in[0,T]}\|r_{\mathrm d}^\tau(h)\|_{1,\log;\tau,\alpha}\le C_{1,T} M_{\alpha,T}\|h\|_{1,\alpha},$$ 
as well as, for every $\varepsilon\in(0,1)$, $$\sup_{\tau\in[0,T]}\|r_{\mathrm d}^\tau(h)\|_{1-\varepsilon,\alpha;\tau}\le C_{\varepsilon,T} M_{\alpha,T}\|h\|_{1,\alpha}.$$
\end{theorem}

\begin{proof}
Fix $\tau\in[0,T]$. By Proposition~\ref{prop:frozen-regularity}, we have
\begin{equation}
\label{eq:dyadic-frozen-defect-proof-regularity}
h^\tau\in ZC^{\gamma}(\Delta^3_\tau;B_\alpha), \quad \|h^\tau\|_{\gamma,\alpha;\tau}\le M_{\alpha,T}\|h\|_{\gamma,\alpha}.
\end{equation}
Let $\gamma\in(0,1)$, we apply
Proposition~\ref{thm:frozen-dyadic-sewing} to $h^\tau$. This yields
\[
r_{\mathrm d}^\tau(h)=\Lambda_\tau^{\mathrm d}(h^\tau)\in C^{\gamma}(\Delta^2_\tau;B_\alpha), \quad \delta r_{\mathrm d}^\tau(h)=h^\tau.
\]
Moreover,
\[
\|r_{\mathrm d}^\tau(h)\|_{\gamma,\alpha;\tau}\le C_\gamma \|h^\tau\|_{\gamma,\alpha;\tau}\overset{(\ref{eq:dyadic-frozen-defect-proof-regularity})}{\le} C_\gamma M_{\alpha,T}\|h\|_{\gamma,\alpha}.
\]
The proof for case $\gamma=1$ is similar to this one.
\end{proof}

\begin{remark}\label{rem:missing-compatibility}
The family $r_{\mathrm d}(h)$ is only a simplexwise construction: for each fixed
$\tau$ it solves the ordinary Sewing problem on $\Delta_\tau$.  It need not
satisfy the transport compatibility condition in
Definition~\ref{def:compatible-frozen-family}, since the dyadic grids depend on
$\tau$.  This missing compatibility will be treated after passing to quotient
classes.
\end{remark}

\section{Quotient semigroup Sewing}\label{sec:quotient-semigroup-sewing}
The dyadic frozen increments constructed in Section~\ref{sec:dyadic-frozen} solve the ordinary Sewing problem on each fixed simplex, but they need not be compatible as actual increments in the terminal parameter.  In this section we show that they nevertheless define canonical compatible quotient objects.  We treat first the subcritical regime \(0<\gamma<1\), and then the critical logarithmic endpoint \(\gamma=1\).

\subsection{Subcritical quotient semigroup Sewing} \label{sec:quotient-semigroup}
We begin by introducing the quotient spaces in which the frozen increments will naturally live.
Let $\tau\in[0,T]$, $\alpha\in\mathbb R$, and $\gamma\in(0,1)$.  Denote by
\[
ZC(\Delta_\tau^2;B_\alpha):=\ker(\delta)\cap C(\Delta^2_\tau;B_\alpha)
\]
the space of ordinary $2$-cocycles on $\Delta_\tau$. For $r,\widetilde r\in C(\Delta^2_\tau;B_\alpha)$, write
\begin{equation}
\label{eq:frozen-equivalence-relation}
r\sim_\tau \widetilde r \quad \Longleftrightarrow \quad r-\widetilde r\in ZC(\Delta^2_\tau;B_\alpha),
\end{equation}
and denote the corresponding equivalence class by $[r]_\tau$. The frozen subcritical quotient space is
\[
\mathfrak C^{\gamma}(\Delta^2_\tau;B_\alpha):=\bigl\{[r]_\tau \mid [r]_\tau\cap C^{\gamma}(\Delta^2_\tau;B_\alpha)\neq\varnothing\bigr\},
\]
equipped with the quotient seminorm
\[
\|[r]_\tau\|_{\mathfrak C_2^{\gamma,\alpha}(\tau)}:=\inf\Bigl\{\| \widetilde r\|_{\gamma,\alpha;\tau} \ \big| \ \widetilde r\in [r]_\tau\cap C^{\gamma}(\Delta^2_\tau;B_\alpha)\Bigr\}.
\]
Quotients are taken modulo all ordinary cocycles, so a semigroup defect naturally defines a class of continuous frozen increments; dyadic normalization yields a 
$\gamma$-H\"older representative, and before defining the semigroup Sewing map we still need restriction and semigroup transport operations on these classes.

\begin{proposition}\label{prop:frozen-quotient-operations}
Let $\gamma\in(0,1)$, $\alpha\in\mathbb R$ and $0\le \tau'\le \tau\le T$.
\begin{enumerate}
\item The restriction map
\begin{equation*}
\operatorname{Res}_{\tau,\tau'}:\mathfrak C^{\gamma}(\Delta^2_\tau;B_\alpha)\to \mathfrak C^{\gamma}(\Delta^2_{\tau'};B_\alpha), \quad [r]_\tau\mapsto\operatorname{Res}_{\tau,\tau'}[r]_\tau:= [r|_{\Delta_{\tau'}^2}]_{\tau'}
\end{equation*}
is well defined and linear. The map satisfies the seminorm bound 
$$\|\operatorname{Res}_{\tau,\tau'}\mathfrak r\|_{\mathfrak C_2^{\gamma,\alpha}(\tau')}\le \|\mathfrak r\|_{\mathfrak C_2^{\gamma,\alpha}(\tau)}.$$
\mlabel{it:froza}

\item The transport map
\begin{equation*}
\mathcal T_{\tau,\tau'}:\mathfrak C^{\gamma}(\Delta^2_{\tau'};B_\alpha)\to \mathfrak C^{\gamma}(\Delta^2_{\tau'};B_\alpha), \quad
[q]_{\tau'}\mapsto \mathcal T_{\tau,\tau'}[q]_{\tau'}:=[S_{\tau-\tau'}q]_{\tau'}
\end{equation*}
is well defined and linear. The map satisfies the seminorm bound 
$$\|\mathcal T_{\tau,\tau'}\mathfrak q\|_{\mathfrak C_2^{\gamma,\alpha}(\tau')}\le M_{\alpha,T}\|\mathfrak q\|_{\mathfrak C_2^{\gamma,\alpha}(\tau')}.$$
\mlabel{it:frozb}
\end{enumerate}
\end{proposition}

\begin{proof}
We only prove Item~(\ref{it:froza}), as the proof for Item~(\ref{it:frozb}) is similar. Suppose that $r\sim_\tau \widetilde r$. Then $r-\widetilde r\in ZC(\Delta^2_\tau;B_\alpha)$, and therefore
$\delta\bigl((r-\widetilde r)|_{\Delta_{\tau'}^2}\bigr)=0$. Hence $r|_{\Delta_{\tau'}^2}\sim_{\tau'} \widetilde r|_{\Delta_{\tau'}^2}$, which proves that the restriction map  is well defined. Linearity is immediate.
Let $\mathfrak r=[r]_\tau\in \mathfrak C^{\gamma}(\Delta^2_\tau;B_\alpha)$ and fix $\varepsilon>0$.
Choose $\widetilde r\in [r]_\tau\cap C^{\gamma}(\Delta^2_\tau;B_\alpha)$ such that
\[
\|\widetilde r\|_{\gamma,\alpha;\tau}\le \|\mathfrak r\|_{\mathfrak C_2^{\gamma,\alpha}(\tau)}+\varepsilon.
\]
Then $\|\widetilde r|_{\Delta_{\tau'}^2}\|_{\gamma,\alpha;\tau'}\le \|\mathfrak r\|_{\mathfrak C_2^{\gamma,\alpha}(\tau)}+\varepsilon$. By taking the infimum over all such $\widetilde r$ gives
\[
\|\operatorname{Res}_{\tau,\tau'}\mathfrak r\|_{\mathfrak C_2^{\gamma,\alpha}(\tau')}\le \|\mathfrak r\|_{\mathfrak C_2^{\gamma,\alpha}(\tau)}.\qedhere
\]
\end{proof}

The quotient classes are still attached to one fixed terminal time. The semigroup Sewing problem, however, is genuinely global in time. The next definition singles out the exact compatibility relation which replaces the transport identity of Definition~\ref{def:compatible-frozen-family}.

\begin{definition}\label{def:compatible-frozen-quotient-families}
Let $\gamma\in(0,1)$ and $\alpha\in\mathbb R$. A {\bf compatible frozen quotient family} is a family $\mathfrak r=(\mathfrak r^\tau)_{\tau\in[0,T]}$ such that
\begin{enumerate}
\item for every $\tau\in[0,T]$, $\mathfrak r^\tau\in \mathfrak C^{\gamma}(\Delta^2_\tau;B_\alpha)$;

\item for every $0\le \tau'\le \tau\le T$, $\operatorname{Res}_{\tau,\tau'}\mathfrak r^\tau=\mathcal T_{\tau,\tau'}\mathfrak r^{\tau'}$;

\item the family is uniformly bounded in quotient seminorm $\sup_{\tau\in[0,T]}\|\mathfrak r^\tau\|_{\mathfrak C_2^{\gamma,\alpha}(\tau)}<\infty$.
\end{enumerate}
We denote by $\mathfrak Q_{\mathrm{fr}}^{\gamma,\alpha}$ the vector space of all compatible frozen quotient families, endowed with the seminorm $$\|\mathfrak r\|_{\mathfrak Q_{\mathrm{fr}}^{\gamma,\alpha}}:=\sup_{\tau\in[0,T]}\|\mathfrak r^\tau\|_{\mathfrak C_2^{\gamma,\alpha}(\tau)}.$$
\end{definition}

We now return to the dyadic representatives constructed in Section~\ref{sec:dyadic-frozen}.  At the level of actual representatives they were not transport-compatible, but the next statement shows that the obstruction disappears after quotienting.

\begin{proposition}\label{prop:dyadic-family-quotient-compatible}
Let $\gamma\in(0,1)$, $\alpha\in\mathbb R$ and $h\in Z\hat C^{\gamma}(\Delta^3_T;B_\alpha)$. Define, for each $\tau\in[0,T]$, $\mathfrak r_{\mathrm d}^\tau(h):=[r_{\mathrm d}^\tau(h)]_\tau$, where $r_{\mathrm d}^\tau(h)$ is the dyadic frozen increment from Definition~\ref{def:dyadic-frozen-defect-family}. Then
\[
\mathfrak r_{\mathrm d}(h):=\bigl(\mathfrak r_{\mathrm d}^\tau(h)\bigr)_{\tau\in[0,T]}\in \mathfrak Q_{\mathrm{fr}}^{\gamma,\alpha}.
\]
Moreover, for every $\tau\in[0,T]$ and every representative $r^\tau\in \mathfrak r_{\mathrm d}^\tau(h)$, one has $\delta r^\tau = h^\tau$. Finally,
\begin{equation}
\label{eq:dyadic-quotient-estimate}
\|\mathfrak r_{\mathrm d}(h)\|_{\mathfrak Q_{\mathrm{fr}}^{\gamma,\alpha}}\le C_\gamma M_{\alpha,T}\|h\|_{\gamma,\alpha}.
\end{equation}
\end{proposition}

\begin{proof}
For each fixed $\tau$, Theorem~\ref{prop:dyadic-frozen-defects} gives
\[
r_{\mathrm d}^\tau(h)\in C^{\gamma}(\Delta^2_\tau;B_\alpha), \quad \delta r_{\mathrm d}^\tau(h)=h^\tau.
\]
Therefore $[r_{\mathrm d}^\tau(h)]_\tau$ belongs to
$\mathfrak C^{\gamma}(\Delta^2_\tau;B_\alpha)$, proving Definition~\ref{def:compatible-frozen-quotient-families} (a). The defect identity $\delta r^\tau = h^\tau$ is now immediate. Indeed, if $r^\tau\in \mathfrak r_{\mathrm d}^\tau(h)$, then $r^\tau-r_{\mathrm d}^\tau(h)\in ZC(\Delta^2_\tau;B_\alpha)$, so $\delta r^\tau=\delta r_{\mathrm d}^\tau(h)=h^\tau$.

We next prove the compatibility relation Definition~\ref{def:compatible-frozen-quotient-families} (b).  Fix $0\le \tau'\le \tau\le T$. On
$\Delta_{\tau'}^3$,  
\[
\delta\bigl(r_{\mathrm d}^\tau(h)|_{\Delta_{\tau'}^2}\bigr)=h^\tau|_{\Delta_{\tau'}^3}.
\]
On the other hand,
\[
\delta\bigl(S_{\tau-\tau'}r_{\mathrm d}^{\tau'}(h)\bigr)=S_{\tau-\tau'}\delta r_{\mathrm d}^{\tau'}(h)=S_{\tau-\tau'}h^{\tau'}.
\]
Now, for $0\le s\le u\le t\le \tau'$, the definition of $h^\tau$ gives
\[
h^\tau_{tus}=S_{\tau-t}h_{tus}=S_{\tau-\tau'}S_{\tau'-t}h_{tus}=S_{\tau-\tau'}h^{\tau'}_{tus}.
\]
Thus
\[
\delta\bigl(r_{\mathrm d}^\tau(h)|_{\Delta_{\tau'}^2}\bigr)=\delta\bigl(S_{\tau-\tau'}r_{\mathrm d}^{\tau'}(h)\bigr).
\]
Their difference is therefore an ordinary cocycle, which exactly means that
\[
\operatorname{Res}_{\tau,\tau'}[r_{\mathrm d}^\tau(h)]_\tau
=
\mathcal T_{\tau,\tau'}[r_{\mathrm d}^{\tau'}(h)]_{\tau'}.
\]
Hence the family is compatible in the sense of Definition~\ref{def:compatible-frozen-quotient-families}.

Finally, since the class seminorm is bounded by the norm of any representative, we obtain from Theorem~\ref{prop:dyadic-frozen-defects} that
\[
\|\mathfrak r_{\mathrm d}^\tau(h)\|_{\mathfrak C_2^{\gamma,\alpha}(\tau)}\le \|r_{\mathrm d}^\tau(h)\|_{\gamma,\alpha;\tau}\le C_\gamma M_{\alpha,T}\|h\|_{\gamma,\alpha}.
\]
Taking the supremum over $\tau\in[0,T]$ proves \eqref{eq:dyadic-quotient-estimate}.
\end{proof}

The previous proposition is the key quotient step.  The dyadic frozen increments need not be compatible as actual increments, but their quotient classes are transport-compatible.  We can therefore define the subcritical quotient semigroup Sewing map by collecting these compatible frozen quotient classes.

\begin{theorem}\label{thm:quotient-semigroup-sewing-subcritical}
Let $\gamma\in(0,1)$ and $\alpha\in\mathbb R$. The assignment
\begin{equation}
\label{eq:quotient-semigroup-sewing-map}
\hat\Lambda_{S,\gamma,\alpha}^{\mathrm q}:Z\hat C^{\gamma}(\Delta^3_T;B_\alpha)\to \mathfrak Q_{\mathrm{fr}}^{\gamma,\alpha},
\quad h\mapsto \hat\Lambda_{S,\gamma,\alpha}^{\mathrm q}h,\quad 
\bigl(\hat\Lambda_{S,\gamma,\alpha}^{\mathrm q}h\bigr)^\tau:=[r_{\mathrm d}^\tau(h)]_\tau,
\end{equation}
is a well defined continuous linear map such that
\begin{enumerate}
\item for every $\tau\in[0,T]$ and every representative $r^\tau\in \bigl(\hat\Lambda_{S,\gamma,\alpha}^{\mathrm q}h\bigr)^\tau$, $\delta r^\tau = h^\tau$; \label{it:quotienta}

\item in particular, for every $0\le s\le u\le t\le T$ and every representative $r^t\in \bigl(\hat\Lambda_{S,\gamma,\alpha}^{\mathrm q}h\bigr)^t$, $(\delta r^t)_{tus}=h_{tus}$;  \label{it:quotientb}

\item for every $h\in Z\hat C^{\gamma}(\Delta^3_T;B_\alpha)$, the uniform estimate $\|\hat\Lambda_{S,\gamma,\alpha}^{\mathrm q}h\|_{\mathfrak Q_{\mathrm{fr}}^{\gamma,\alpha}}\le C_{\gamma, T} M_{\alpha,T}\|h\|_{\gamma,\alpha}$ holds.  \label{it:quotientc}
\end{enumerate}
\end{theorem}

\begin{proof}
For each $\tau$, the map $h\longmapsto h^\tau=\Theta_\tau h$ is linear, and the dyadic frozen Sewing map
\[
\Lambda_\tau^{\mathrm d}:ZC^{\gamma}(\Delta^3_\tau;B_\alpha)\to C^{\gamma}(\Delta^2_\tau;B_\alpha)
\]
from Proposition~\ref{thm:frozen-dyadic-sewing} is also linear. Therefore $h\longmapsto r_{\mathrm d}^\tau(h)=\Lambda_\tau^{\mathrm d}(h^\tau)$ is linear for every $\tau$, and so is the family-valued map
\eqref{eq:quotient-semigroup-sewing-map}. Proposition~\ref{prop:dyadic-family-quotient-compatible} shows that
\[
\bigl([r_{\mathrm d}^\tau(h)]_\tau\bigr)_{\tau\in[0,T]}\in \mathfrak Q_{\mathrm{fr}}^{\gamma,\alpha},
\]
so the map is well defined. The same proposition also gives Item~(\ref{it:quotientc}), hence continuity.

It remains to prove Item~(\ref{it:quotienta}) and Item~(\ref{it:quotientb}). Let $\tau\in[0,T]$ and $r^\tau\in \bigl(\hat\Lambda_{S,\gamma,\alpha}^{\mathrm q}h\bigr)^\tau$. By
definition of the quotient class, $r^\tau-r_{\mathrm d}^\tau(h)$ is an ordinary
cocycle, so
\[
\delta r^\tau=\delta r_{\mathrm d}^\tau(h)=h^\tau.
\]
This proves Item~(\ref{it:quotienta}). Finally, setting $\tau=t$ and using $h^t_{tus}=S_{t-t}h_{tus}=h_{tus}$ yields Item~(\ref{it:quotientb}).
\end{proof}

The preceding theorem constructs the canonical subcritical quotient object.  The next corollary shows that, whenever an actual semigroup increment exists, its frozen quotient classes coincide with the dyadic quotient classes constructed above.

\begin{coro}\label{cor:quotient-semigroup-intrinsic}
Let $\gamma\in(0,1)$, $\alpha\in\mathbb R$ and 
$h\in Z\hat C^{\gamma}(\Delta^3_T;B_\alpha)$.  If
\begin{equation}
\label{eq:arbitrary-semigroup-primitive}
B\in C(\Delta^2_T;B_\alpha),\quad \hat\delta B = h,
\end{equation}
then, for every $\tau\in[0,T]$,
\begin{equation}
\label{eq:quotient-semigroup-intrinsic-formula}
\bigl(\hat\Lambda_{S,\gamma,\alpha}^{\mathrm q}h\bigr)^\tau=[\Theta_\tau B]_\tau .
\end{equation}
In particular,
\begin{equation}
\label{eq:quotient-semigroup-canonical-primitive}
\bigl(\hat\Lambda_{S,\gamma,\alpha}^{\mathrm q}h\bigr)^\tau=[b^\tau(h)]_\tau,\quad b^\tau(h)_{ts}:=-h^\tau_{ts0}.
\end{equation}
Consequently, the quotient class $\bigl(\hat\Lambda_{S,\gamma,\alpha}^{\mathrm q}h\bigr)^\tau$ is independent of both the chosen semigroup increment \(B\) and the dyadic normalization used in Section~\ref{sec:dyadic-frozen}.
\end{coro}

\begin{proof}
Fix $\tau\in[0,T]$. By Proposition~\ref{prop:frozen-intertwining} and \eqref{eq:arbitrary-semigroup-primitive},
\[
\delta(\Theta_\tau B)=\Theta_\tau(\hat\delta B)=\Theta_\tau h=h^\tau.
\]
On the other hand, Theorem~\ref{prop:dyadic-frozen-defects} gives $\delta r_{\mathrm d}^\tau(h)=h^\tau$. Hence $\delta\bigl(\Theta_\tau B-r_{\mathrm d}^\tau(h)\bigr)=0$, which means exactly that
\[
[\Theta_\tau B]_\tau=[r_{\mathrm d}^\tau(h)]_\tau.
\]
By definition of $\hat\Lambda_{S,\gamma,\alpha}^{\mathrm q}$, this proves
\eqref{eq:quotient-semigroup-intrinsic-formula}.
Let \(B^h_{ts}:=-h_{ts0}\).  By Lemma~\ref{prop:exactness}, \(\hat\delta B^h=h\), and \(\Theta_\tau B^h=b^\tau(h)\). Taking $B=B^h$ yields \eqref{eq:quotient-semigroup-canonical-primitive}. The independence of the chosen semigroup increment follows by applying the same argument to two semigroup increments with the same \(\hat\delta\)-defect.  The independence of the dyadic normalization follows because any other dyadic representative with
ordinary defect \(h^\tau\) differs from \(r_{\mathrm d}^\tau(h)\) by an ordinary cocycle.
\end{proof}

\subsection{Critical logarithmic quotient semigroup Sewing} \label{sec:critical-semigroup}
At the critical endpoint \(\gamma=1\), the algebraic quotient construction is unchanged, but the analytic estimate changes because ordinary frozen Sewing only gives logarithmic endpoint control.  We therefore treat this case separately, using logarithmic quotient spaces in place of the subcritical Hölder spaces.

We keep the cocycle equivalence relation \eqref{eq:frozen-equivalence-relation}. The only new point is that the regular representatives are now measured in the critical logarithmic class. Let $\tau\in[0,T]$ and $\alpha\in\mathbb R$. We define
\[
\mathfrak C_{\log}^{1}(\Delta^2_\tau;B_\alpha):=\bigl\{[r]_\tau \mid [r]_\tau\cap C_{\log}^{1}(\Delta^2_\tau;B_\alpha)\neq\varnothing\bigr\}.
\]
For $\mathfrak r=[r]_\tau\in \mathfrak C_{\log}^{1}(\Delta^2_\tau;B_\alpha)$, we define the quotient logarithmic seminorm by
\[
\|\mathfrak r\|_{\mathfrak C_{2,\log}^{1,\alpha}(\tau)}:=\inf\Bigl\{\|\widetilde r\|_{1,\log;\tau,\alpha} \ \big|\  \widetilde r\in [r]_\tau\cap C_{\log}^{1}(\Delta^2_\tau;B_\alpha)
\Bigr\}.
\]

The equivalence relation is the same as in the subcritical case; only the analytic representative space is replaced by the logarithmic one.  Restriction and semigroup transport therefore descend to the critical quotient in the same way.

\begin{proposition}\label{prop:critical-frozen-quotient-operations}
Let $\alpha\in\mathbb R$ and $0\le \tau'\le \tau\le T$.
\begin{enumerate}
\item The restriction map
\begin{equation*}
\operatorname{Res}_{\tau,\tau'}:\mathfrak C^{1}_{\mathrm{log}}(\Delta^2_\tau;B_\alpha)\to \mathfrak C^{1}_{\mathrm{log}}(\Delta^2_{\tau'};B_\alpha), \quad [r]_\tau\mapsto\operatorname{Res}_{\tau,\tau'}[r]_\tau:= [r|_{\Delta_{\tau'}^2}]_{\tau'}
\end{equation*}
is well defined and linear. The map satisfies the seminorm bound 
$$\|\operatorname{Res}_{\tau,\tau'}\mathfrak r\|_{\mathfrak C_{2, \mathrm{log}}^{1,\alpha}(\tau')}\le \|\mathfrak r\|_{\mathfrak C_{2, \mathrm{log}}^{1,\alpha}(\tau)}.$$

\item The transport map
\begin{equation*}
\mathcal T_{\tau,\tau'}:\mathfrak C^{1}_{\mathrm{log}}(\Delta^2_{\tau'};B_\alpha)\to \mathfrak C^{1}_{\mathrm{log}}(\Delta^2_{\tau'};B_\alpha), \quad
[q]_{\tau'}\mapsto \mathcal T_{\tau,\tau'}[q]_{\tau'}:=[S_{\tau-\tau'}q]_{\tau'}
\end{equation*}
is well defined and linear. The map satisfies the seminorm bound 
$$\|\mathcal T_{\tau,\tau'}\mathfrak q\|_{\mathfrak C_{2, \mathrm{log}}^{1,\alpha}(\tau')}\le M_{\alpha,T}\|\mathfrak q\|_{\mathfrak C_{2, \mathrm{log}}^{1,\alpha}(\tau')}.$$
\end{enumerate}
\end{proposition}

\begin{proof}
This is the same argument as in Proposition~\ref{prop:frozen-quotient-operations}. Restriction and semigroup transport preserve ordinary cocycles, so both maps are
well defined on quotient classes.  The logarithmic quotient seminorm estimates follow from the corresponding representative estimates and the semigroup bound.
\end{proof}

The previous proposition shows that endpoint quotient classes behave well under restriction and transport. We therefore package them into transport-compatible families, as in the subcritical section.

\begin{definition}\label{def:compatible-critical-quotient-families}
Let  $\alpha\in\mathbb R$. A compatible frozen quotient family is a family $\mathfrak r=(\mathfrak r^\tau)_{\tau\in[0,T]}$ such that
\begin{enumerate}
\item for every $\tau\in[0,T]$, $\mathfrak r^\tau\in \mathfrak C^{1}_{\mathrm{log}}(\Delta^2_\tau;B_\alpha)$;

\item for every $0\le \tau'\le \tau\le T$, $\operatorname{Res}_{\tau,\tau'}\mathfrak r^\tau=\mathcal T_{\tau,\tau'}\mathfrak r^{\tau'}$;

\item the family is uniformly bounded in quotient seminorm $\sup_{\tau\in[0,T]}\|\mathfrak r^\tau\|_{\mathfrak C_{2, \mathrm{log}}^{1,\alpha}(\tau)}<\infty$.
\end{enumerate}
Denote by $\mathfrak Q_{\mathrm{fr}, \mathrm{log}}^{1,\alpha}$ the vector space of all compatible frozen quotient families, endowed with the seminorm $$\|\mathfrak r\|_{\mathfrak Q_{\mathrm{fr}, \mathrm{log}}^{1,\alpha}}:=\sup_{\tau\in[0,T]}\|\mathfrak r^\tau\|_{\mathfrak C_{2, \mathrm{log}}^{1,\alpha}(\tau)}.$$
\end{definition}

We now apply the critical part of Proposition~\ref{thm:frozen-dyadic-sewing} to the frozen defects \(h^\tau=\Theta_\tau h\).  The resulting representatives need not be transport-compatible as actual increments, but their quotient classes are compatible.

\begin{proposition}\label{prop:critical-dyadic-family-quotient-compatible}
Let $\alpha\in\mathbb R$ and $h\in Z\hat C^{1}(\Delta^3_T;B_\alpha)$. For each $\tau\in[0,T]$, define $\mathfrak r_{\mathrm d, \mathrm {log}}^\tau(h):=[r_{\mathrm d}^\tau(h)]_\tau$, where $r_{\mathrm d}^\tau(h)$ is the dyadic frozen increment from Definition~\ref{def:dyadic-frozen-defect-family}. Then
\[
\mathfrak r_{\mathrm d, \mathrm {log}}(h):=\bigl(\mathfrak r_{\mathrm d, \mathrm {log}}^\tau(h)\bigr)_{\tau\in[0,T]}\in \mathfrak Q_{\mathrm{fr}, \mathrm {log}}^{1,\alpha}.
\]
Moreover, for every $\tau\in[0,T]$ and every representative $r^\tau\in \mathfrak r_{\mathrm d, \mathrm {log}}^\tau(h)$, one has $\delta r^\tau = h^\tau$. Finally,
\begin{equation*}
\|\mathfrak r_{\mathrm d, \mathrm {log}}(h)\|_{\mathfrak Q_{\mathrm{fr}, \mathrm {log}}^{1,\alpha}}\le C_T M_{\alpha,T}\|h\|_{1,\alpha}.
\end{equation*}
\end{proposition}

\begin{proof}
The proof is the same as that of Proposition~\ref{prop:dyadic-family-quotient-compatible}, with the critical logarithmic quotient norm in place of the subcritical Hölder quotient norm.  The endpoint estimate from Section~\ref{sec:dyadic-frozen} gives the stated membership and bound, while the identity $\delta r_{\mathrm d}^\tau(h)=h^\tau$ gives the defect statement.
For \(\tau'\le\tau\), the identity
\[
h^\tau|_{\Delta_{\tau'}^3}=S_{\tau-\tau'}h^{\tau'}
\]
follows directly from the definition of freezing.  Hence $r_{\mathrm d}^\tau(h)|_{\Delta_{\tau'}^2} $ and $S_{\tau-\tau'}r_{\mathrm d}^{\tau'}(h)$ have the same ordinary defect, and therefore define the same class in $\mathfrak C^{1}_{\mathrm{log}}(\Delta^2_{\tau'};B_\alpha)$.  This is exactly the quotient compatibility.
\end{proof}

The logarithmic endpoint class also controls every subcritical exponent below $1$, a fact that will be useful when combined with Banach-scale smoothing.

\begin{coro}\label{cor:critical-dyadic-family-subcritical-traces}
Let $\alpha\in\mathbb R$, $h\in Z\hat C^{1}(\Delta^3_T;B_\alpha)$ and fix $\varepsilon\in(0,1)$. For each $\tau\in[0,T]$, define $\mathfrak r_{\mathrm d,\varepsilon}^\tau(h):=[r_{\mathrm d}^\tau(h)]_\tau$. Then
\[
\mathfrak r_{\mathrm d,\varepsilon}(h):=\bigl(\mathfrak r_{\mathrm d,\varepsilon}^\tau(h)\bigr)_{\tau\in[0,T]}\in\mathfrak Q_{\mathrm{fr}}^{1-\varepsilon,\alpha}, \quad \|\mathfrak r_{\mathrm d,\varepsilon}(h)\|_{\mathfrak Q_{\mathrm{fr}}^{1-\varepsilon,\alpha}}\le C_{\varepsilon,T}M_{\alpha,T}\|h\|_{1,\alpha}.
\]
\end{coro}

\begin{proof}
By the endpoint logarithmic estimate and the elementary embedding
\[
C_{\log}^{1}(\Delta^2_{\tau};B_\alpha)\hookrightarrow C^{1-\varepsilon}(\Delta^2_{\tau};B_\alpha),
\]
we have
\[
[r_{\mathrm d}^\tau(h)]_\tau\in \mathfrak C^{1-\varepsilon}(\Delta^2_{\tau};B_\alpha), \quad \|[r_{\mathrm d}^\tau(h)]_\tau\|_{\mathfrak C_2^{1-\varepsilon,\alpha}(\tau)}\le C_{\varepsilon,T}M_{\alpha,T}\|h\|_{1,\alpha}.
\]
The quotient compatibility is the same as in Proposition~\ref{prop:critical-dyadic-family-quotient-compatible}, since the equivalence relation is still modulo ordinary cocycles. Taking the supremum over $\tau\in[0,T]$ gives the result.
\end{proof}

We now package the endpoint construction as a semigroup Sewing map with values in the frozen logarithmic quotient families of Definition~\ref{def:compatible-critical-quotient-families}.

\begin{theorem}\label{thm:critical-quotient-semigroup-sewing}
Let $\alpha\in\mathbb R$. The assignment
\begin{equation}
\label{eq:critical-quotient-semigroup-sewing-map}
\hat\Lambda_{S,1,\alpha}^{\mathrm q}:Z\hat C^{1}(\Delta^3_T;B_\alpha)\to \mathfrak Q_{\mathrm{fr},\mathrm{log}}^{1,\alpha},
\quad h\mapsto \hat\Lambda_{S,1,\alpha}^{\mathrm q}h,\quad 
\bigl(\hat\Lambda_{S,1,\alpha}^{\mathrm q}h\bigr)^\tau:=[r_{\mathrm d}^\tau(h)]_\tau,
\end{equation}
is a well-defined continuous linear map. It satisfies the following properties
\begin{enumerate}
\item for every $\tau\in[0,T]$ and every representative $r^\tau\in \bigl(\hat\Lambda_{S,1,\alpha}^{\mathrm q}h\bigr)^\tau$, $\delta r^\tau = h^\tau$;

\item in particular, for every $0\le s\le u\le t\le T$ and every representative $r^t\in \bigl(\hat\Lambda_{S,1,\alpha}^{\mathrm q}h\bigr)^t$, $(\delta r^t)_{tus}=h_{tus}$;

\item for every $h\in Z\hat C^{1}(\Delta^3_T;B_\alpha)$, the uniform estimate $\|\hat\Lambda_{S,1,\alpha}^{\mathrm q}h\|_{\mathfrak Q_{\mathrm{fr}, \mathrm{log}}^{1,\alpha}}\le C_{T} M_{\alpha,T}\|h\|_{1,\alpha}$ holds.
\end{enumerate}
\end{theorem}

\begin{proof}
The proof is the same as in the subcritical case, with $\mathfrak Q_{\mathrm{fr},\log}^{1,\alpha}$ replacing $\mathfrak Q_{\mathrm{fr}}^{\gamma,\alpha}$.  Proposition~\ref{prop:critical-dyadic-family-quotient-compatible} shows that the family
\[
\bigl([r_{\mathrm d}^{\tau}(h)]_\tau\bigr)_{\tau\in[0,T]}
\]
belongs to the critical compatible quotient space and satisfies the stated estimate.  Linearity follows from the linearity of freezing and of the dyadic ordinary Sewing operator, and the defect identity follows from
$
\delta r_{\mathrm d}^{\tau}(h)=h^\tau .
$
The terminal identity is obtained by taking $\tau=t$, since $h^t_{tus}=h_{tus}$.
\end{proof}

The endpoint quotient map is the logarithmic analogue of the subcritical map from Section~\ref{sec:quotient-semigroup}.  The next corollary identifies its quotient
class directly through any semigroup increment of the defect.

\begin{coro}\label{cor:critical-quotient-semigroup-intrinsic}
Let $\alpha\in\mathbb R$ and $h\in Z\hat C^{1}(\Delta^3_T;B_\alpha)$. Assume that $B\in C(\Delta^2_T;B_\alpha)$ and $\hat\delta B = h$. Then, for every $\tau\in[0,T]$, 
\begin{equation*}
\bigl(\hat\Lambda_{S,1,\alpha}^{\mathrm q}h\bigr)^\tau=[\Theta_\tau B]_\tau\in \mathfrak C^{1}_{\mathrm{log}}(\Delta^2_{\tau};B_\alpha).
\end{equation*}
Moreover, for every $\varepsilon\in(0,1)$, the same underlying class satisfies
\begin{equation*}
\bigl(\hat\Lambda_{S,1,\alpha}^{\mathrm q}h\bigr)^\tau=[\Theta_\tau B]_\tau\in \mathfrak C^{1-\varepsilon}(\Delta^2_{\tau};B_\alpha).
\end{equation*}
In particular,
\begin{equation*}
\bigl(\hat\Lambda_{S,1,\alpha}^{\mathrm q}h\bigr)^\tau=[b^\tau(h)]_\tau, \quad b^\tau(h)_{ts}:=-h^\tau_{ts0}.
\end{equation*}
Consequently, the critical quotient class is independent both of the chosen semigroup increment $B$ and of the dyadic normalization used in Section~\ref{sec:dyadic-frozen}.
\end{coro}

\begin{proof}
Let \(B\in C(\Delta_T^2;B_\alpha)\) satisfy \(\hat\delta B=h\).  Freezing gives $\delta(\Theta_\tau B)=h^\tau$. By construction, $\delta r_{\mathrm d}^{\tau}(h)=h^\tau$ .
Hence \(\Theta_\tau B-r_{\mathrm d}^{\tau}(h)\in ZC(\Delta^2_\tau;B_\alpha)\), and
therefore
\[
(\hat\Lambda_{S,1,\alpha}^{\mathrm q}h)^\tau
=
[r_{\mathrm d}^{\tau}(h)]_\tau
=
[\Theta_\tau B]_\tau .
\]
Since \(r_{\mathrm d}^{\tau}(h)\) belongs to the logarithmic endpoint class, and $$C_{\log}^{1}(\Delta_\tau^2;B_\alpha)\hookrightarrow C^{1-\varepsilon}(\Delta_\tau^2;B_\alpha)$$ for every $\varepsilon\in(0,1)$, the same quotient class also belongs to $\mathfrak C^{1-\varepsilon}(\Delta_\tau^2;B_\alpha)$.
Taking \(B=B^h\), where \(B^h_{ts}:=-h_{ts0}\), gives
\[
(\hat\Lambda_{S,1,\alpha}^{\mathrm q}h)^\tau=[b^\tau(h)]_\tau, \quad b^\tau(h)_{ts}:=-h^\tau_{ts0}.
\]
The independence of the chosen two increment and of the dyadic normalization follows because any two representatives with the same ordinary defect differ by an ordinary cocycle.
\end{proof}

\begin{remark}\label{rem:quotient-versus-representatives}
We complete the subcritical and critical semigroup Sewing constructions at the quotient level.  The canonical objects are compatible frozen quotient classes, not necessarily compatible actual representatives.  The next section addresses the additional analytic conditions under which such quotient classes can be lifted to concrete semigroup increments.
\end{remark}

\section{Representative selection and actual semigroup Sewing}
\label{sec:representative-selection-sewing}

In this section, we introduce a scale-dependent selection mechanism and use it to construct actual semigroup increments from admissible defects.

\subsection{Scale-dependent defect splitting}
\label{sec:scale-splitting}

The splitting is imposed only along the residuals generated by the construction, not on arbitrary vectors of $B_\alpha$. Let $H\in Z\hat C^{\gamma}(\Delta_T^3;B_\alpha)$. By Lemma~\ref{prop:exactness}, define the semigroup increment $B^H_{ts}:=-H_{ts0}$, so that $\hat\delta B^H=H$.  In the dyadic recursion, if $b=(a+c)/2$, the semigroup identity requires
\[
R_{ca}=R_{cb}+S_{c-b}R_{ba}+H_{cba}.
\]
Thus the splitting is imposed only on the residual vector $R_{ca}-H_{cba}$, not on arbitrary vectors of $B_\alpha$.

\begin{definition}
\label{def:defect-dependent-splitting}
Let $0<\gamma\le1$ and  $H\in Z\hat C^{\gamma}(\Delta_T^3;B_\alpha)$.  A defect-dependent scale splitting for $H$ at level $B_\alpha$ is a recursive rule $\mathsf s$ which acts on the residuals generated by the dyadic construction as follows.  If a dyadic interval $[a,c]$ has midpoint $b=(a+c)/2$ and carries a residual $Q_{ca}\in B_\alpha$, then $\mathsf s$ assigns two child residuals $Q^{\mathsf s}_{cb}, Q^{\mathsf s}_{ba} \in B_\alpha$, satisfying
\begin{equation}
\label{eq:scale-splitting-identity}
Q_{ca}=Q^{\mathsf s}_{cb}+S_{c-b}Q^{\mathsf s}_{ba}+H_{cba}.
\end{equation}
If \(0<\gamma<1\), the rule $\mathsf s$ is called {\bf subcritical admissible} if there exist constants $0\le\theta<2^{-\gamma}$ and $C_H\ge0$ such that, for every residual
$Q_{ca}$ generated on a dyadic interval $[a,c]$ of length $L=c-a$,
\begin{equation}
\label{eq:subcritical-scale-splitting-estimate}
\max\bigl\{\norm{Q^{\mathsf s}_{cb}}_{B_\alpha}, \norm{Q^{\mathsf s}_{ba}}_{B_\alpha}\bigr\}\le \theta\norm{Q_{ca}}_{B_\alpha}+C_H L^\gamma .
\end{equation}
If $\gamma=1$, the rule $\mathsf s$ is called {\bf critical admissible} if there exist constants $C_H,C_\ast,\eta>0$ such that, for every residual $Q_{ca}$ generated on a dyadic interval $[a,c]$ of length $L=c-a$,
\begin{equation}
\label{eq:critical-scale-splitting-estimate}
\max\bigl\{\norm{Q^{\mathsf s}_{cb}}_{B_\alpha}, \norm{Q^{\mathsf s}_{ba}}_{B_\alpha}\bigr\}\le\left(\frac12+C_\ast L^\eta\right) \norm{Q_{ca}}_{B_\alpha}+C_H L .
\end{equation}
\end{definition}

\begin{remark}
The splitting rule is imposed only on the residuals generated by the dyadic construction, not on arbitrary vectors of $B_\alpha$.  The two admissibility conditions reflect the dyadic estimates needed later: in the subcritical case $\theta<2^{-\gamma}$ gives contraction faster than the target scale $L^\gamma$, while at $\gamma=1$ the coefficient is borderline, equal to $1/2$ up to a summable scale error, producing the logarithmic loss.
\end{remark}

To turn the defect-wise splitting condition into a linear Sewing map, we now fix a linear class of defects on which the same splitting rule and uniform estimates are available.

\begin{definition}
\label{def:admissible-selection-domain}
Let $\alpha\in\R$ and fix a linear splitting rule $\mathsf s$. 
\begin{enumerate}
\item For $0<\gamma<1$, a {\bf subcritical admissible selection domain} is a linear subspace $\D^{\mathsf s}_{\gamma}(\Delta_T^3;B_\alpha)\subset Z\hat C^{\gamma}(\Delta_T^3;B_\alpha)$ such that, for every $H\in\D^{\mathsf s}_{\gamma}(\Delta_T^3;B_\alpha)$, the dyadic recursion
is well defined, \eqref{eq:scale-splitting-identity} holds at each splitting step, and \eqref{eq:subcritical-scale-splitting-estimate} holds with $\theta<2^{-\gamma}$ and
$C_H\le C_{\mathsf s}\norm{H}_{\gamma,\alpha}$.

\item At $\gamma=1$, a {\bf critical admissible selection domain} is a linear subspace $$\D^{\mathsf s}_{1}(\Delta_T^3;B_\alpha)\subset Z\hC^{1}(\Delta_T^3;B_\alpha)$$
such that, for every $H\in\D^{\mathsf s}_{1}(\Delta_T^3;B_\alpha)$, the dyadic recursion is well defined, \eqref{eq:scale-splitting-identity} holds at each splitting step,
and \eqref{eq:critical-scale-splitting-estimate} holds with fixed $C_\ast,\eta>0$ and $C_H\le C_{\mathsf s}\norm{H}_{1,\alpha}$.
\end{enumerate}
In both cases we require the residual continuity criterion
\begin{equation}
\label{eq:residual-continuity-criterion}
\lim_{h\downarrow0}\sup_{\substack{s\le t,\ s,t\in\D\\ t+h\le T}}\norm{(S_h-\mathrm{Id})R^{\mathsf s}_{ts}}_{B_\alpha}=0
\end{equation}
for every $H$ in the corresponding domain.
\end{definition}

\begin{remark}
\label{rem:individual-versus-linear}
For a single defect, the construction only gives a representative with constants depending on that defect.  A continuous linear Sewing map $\hat\Lambda^{\mathsf s}$ is obtained only after fixing a linear splitting rule on a linear admissible selection domain.  This is why Definition~\ref{def:admissible-selection-domain} is needed.
\end{remark}

\subsection{Direct dyadic construction of semigroup Sewing maps}
\label{sec:direct-dyadic}

We now perform the representative-selection step directly in the semigroup complex, producing strict semigroup representatives rather than frozen quotient classes.
Let $H\in Z\hat C^{\gamma}(\Delta_T^3;B_\alpha)$ and let $\D_n:=\{kT2^{-n}:0\le k\le 2^n\}$, $\D:=\bigcup_{n\ge0}\D_n$.  The recursion constructs a dyadic path $I^{\mathsf s}(H)$, and the selected residual increment is always understood as
\begin{equation}
\label{eq:residual-from-path}
R^{\mathsf s}_{ts}:=B^H_{ts}-I^{\mathsf s}_t+S_{t-s}I^{\mathsf s}_s, \quad s,t\in\D,\ s\le t.
\end{equation}
Set $I^{\mathsf s}_0(H):=0$ and $I^{\mathsf s}_T(H):=B^H_{T0}$. Then $R^{\mathsf s}_{T0}=0$.  Suppose that $I^{\mathsf s}$ has been defined at the endpoints $a<c$ of a dyadic interval and let $b=(a+c)/2$.  The parent residual is 
\[
Q_{ca}:=R^{\mathsf s}_{ca}.
\]
Apply the splitting rule $\mathsf s$ to obtain child residuals $Q^{\mathsf s}_{cb}$ and $Q^{\mathsf s}_{ba}$ satisfying \eqref{eq:scale-splitting-identity}.  Define the midpoint value by
\begin{equation}
\label{eq:midpoint-path-scale}
I^{\mathsf s}_b(H):=B^H_{ba}+S_{b-a}I^{\mathsf s}_a(H)-Q^{\mathsf s}_{ba}.
\end{equation}
The next lemma records the algebraic consequence of this choice.

\begin{lemma}
\label{lem:local-compatibility}
For every dyadic parent interval $[a,c]$ with midpoint $b$, the residuals produced by the dyadic recursion satisfy
\begin{equation}
\label{eq:local-compatibility}
R^{\mathsf s}_{ca}=R^{\mathsf s}_{cb}+S_{c-b}R^{\mathsf s}_{ba}+H_{cba}.
\end{equation}
Consequently, $\hdd R^{\mathsf s}=H$ on dyadic triples.
\end{lemma}

\begin{proof}
The midpoint value \(I_b^{\mathsf s}\) was chosen so that the child residuals computed from \eqref{eq:residual-from-path} coincide with those produced by the
splitting rule; hence \eqref{eq:scale-splitting-identity} gives \eqref{eq:local-compatibility}.  Since $R^{\mathsf s}=B^H-\hdd I^{\mathsf s}$, $\hdd B^H=H$ and $\hdd^2=0$, we obtain $\hdd R^{\mathsf s}=H$.
\end{proof}

We first treat the case $0<\gamma<1$.  The strict contraction condition $\theta<2^{-\gamma}$ in \eqref{eq:subcritical-scale-splitting-estimate} yields a genuine H\"older bound.

\begin{theorem}
\label{thm:subcritical-dyadic-sewing}
Let $0<\gamma<1$, $\alpha\in\R$ and let $\D^{\mathsf s}_{\gamma}(\Delta_T^3;B_\alpha)$ be a subcritical admissible selection domain in the sense of Definition~\ref{def:admissible-selection-domain}.  Then the dyadic recursion above extends uniquely by continuity to a continuous linear map
\begin{equation*}
\hat\Lambda^{\mathsf s}_{S,\gamma,\alpha}:\D^{\mathsf s}_{\gamma}(\Delta_T^3;B_\alpha) \to C^{\gamma}(\Delta_T^2;B_\alpha), \quad H\mapsto \hat\Lambda^{\mathsf s}_{S,\gamma,\alpha}H:= R^{\mathsf s}(H),
\end{equation*}
satisfying
\begin{equation}
\label{eq:subcritical-right-inverse}
\hdd\hat\Lambda^{\mathsf s}_{S,\gamma,\alpha}H=H.
\end{equation}
Moreover,
\begin{equation}
\label{eq:subcritical-estimate}
\norm{\hat\Lambda^{\mathsf s}_{S,\gamma,\alpha}H}_{\gamma,\alpha}\le C_{\gamma,\theta,T,S,\mathsf s}\norm{H}_{\gamma,\alpha}.
\end{equation}
For an individual defect outside a fixed linear domain, the same estimate holds with the right-hand side replaced by $C_{\gamma,\theta,T,S}(\norm{H}_{\gamma,\alpha}+C_H)$.
\end{theorem}

\begin{proof}
We first prove the dyadic estimate. Let
\[
M_n:=\max_{0\le k<2^n}\norm{R^{\mathsf s}_{t_{k+1}^n,t_k^n}}_{B_\alpha}, \quad  t_k^n:=kT2^{-n}.
\]
At level $0$, $R^{\mathsf s}_{T0}=0$, so $M_0=0$.  Let $[a,c]$ be a dyadic interval of length $L=T2^{-n}$ and midpoint $b$.  By \eqref{eq:subcritical-scale-splitting-estimate}, each child residual is bounded by $\theta\norm{R^{\mathsf s}_{ca}}_{B_\alpha}+C_H L^\gamma$. Hence
\begin{equation*}
M_{n+1}\le \theta M_n+C_H(T2^{-n})^\gamma.
\end{equation*}
Iterating gives
\[
M_n\le C_H T^\gamma
\sum_{j=0}^{n-1}\theta^{n-1-j}2^{-j\gamma}.
\]
Since $\theta<2^{-\gamma}$, the sum is bounded by a constant times $2^{-n\gamma}$.  Therefore
\begin{equation}
\label{eq:consecutive-bound-scale}
M_n\le C_{\gamma,\theta}C_H(T2^{-n})^\gamma.
\end{equation}

Let $s<t$ be dyadic.  Decompose $[s,t]$ into the disjoint union of maximal dyadic intervals $I_j$.  Repeated use of \eqref{eq:local-compatibility}, the semigroup bound
and the definition of the $C^{\gamma}(\Delta_T^3;B_\alpha)$ norm yields
\[
\norm{R^{\mathsf s}_{ts}}_{B_\alpha}\le C_{S,T}\sum_j \norm{R^{\mathsf s}_{I_j}}_{B_\alpha}+C_{S,T}\norm{H}_{\gamma,\alpha}\sum_j |I_j|^\gamma.
\]
A standard dyadic decomposition has at most two intervals at each scale, hence
\[
\sum_j |I_j|^\gamma\le C_\gamma |t-s|^\gamma.
\]
Using \eqref{eq:consecutive-bound-scale}, we obtain
\begin{equation}
\label{eq:dyadic-holder-bound-scale}
\norm{R^{\mathsf s}_{ts}}_{B_\alpha}\le C_{\gamma,\theta,T,S}\bigl(\norm{H}_{\gamma,\alpha}+C_H\bigr)|t-s|^\gamma
\end{equation}
for all dyadic $s<t$.

We now extend the dyadic residuals to all times.  Apart from the usual Hölder control, one must handle the semigroup term created by varying the upper endpoint; this is precisely the role of \eqref{eq:residual-continuity-criterion}. For dyadic triples $s\le u\le t$, Lemma~\ref{lem:local-compatibility} gives
\begin{equation*}
R^{\mathsf s}_{ts}=R^{\mathsf s}_{tu}+S_{t-u}R^{\mathsf s}_{us}+H_{tus}.
\end{equation*}
Thus, if $s\le t\le t'$ are dyadic, then
\[
R^{\mathsf s}_{t's}-R^{\mathsf s}_{ts}=R^{\mathsf s}_{t't}+H_{t'ts}+(S_{t'-t}-\mathrm{Id})R^{\mathsf s}_{ts},
\]
and the first two terms are small by \eqref{eq:dyadic-holder-bound-scale} and the H\"older bound on $H$, while the last term is small uniformly by \eqref{eq:residual-continuity-criterion}.  Similarly, if $s'\le s\le t$ are dyadic, then
\[
R^{\mathsf s}_{ts'}-R^{\mathsf s}_{ts}=S_{t-s}R^{\mathsf s}_{ss'}+H_{tss'},
\]
which is small by \eqref{eq:dyadic-holder-bound-scale} and the bound on $H$.
These two endpoint comparison estimates imply that the dyadic values are uniformly Cauchy under dyadic approximation of any point of $\Delta_T^2$. Hence $R^{\mathsf s}$ has a unique continuous extension to $\Delta_T^2$. The estimate \eqref{eq:dyadic-holder-bound-scale} passes to the limit, and the domain bound $C_H\le C_{\mathsf s}\|H\|_{\gamma,\alpha}$ gives \eqref{eq:subcritical-estimate} and the stated operator estimate.  Linearity follows from the linearity of $H\mapsto B^H$ and of the fixed splitting rule on $\D^{\mathsf s}_{\gamma,\alpha}$.
By Lemma~\ref{lem:local-compatibility}, $\hdd R^{\mathsf s}=H$ on dyadic triples.  After the extension, both sides are continuous on $\Delta_T^3$; hence passing to the limit along dyadic triples gives \eqref{eq:subcritical-right-inverse} for all $0\le s\le u\le t\le T$.
\end{proof}

The endpoint $\gamma=1$ is borderline.  The coefficient $1/2$ in \eqref{eq:critical-scale-splitting-estimate} leads to a logarithmic loss, while the perturbation $C_\ast L^\eta$ is summable along dyadic scales.

\begin{theorem}
\label{thm:critical-dyadic-sewing}
Let $\alpha\in\R$ and let $\D^{\mathsf s}_{1}(\Delta_T^3;B_\alpha)$ be a critical admissible selection domain in the sense of Definition~\ref{def:admissible-selection-domain}.  Then the dyadic recursion above extends uniquely by continuity to a continuous linear map
\begin{equation*}
\hat\Lambda^{\mathsf s}_{S,1,\alpha}:\D^{\mathsf s}_{1}(\Delta_T^3;B_\alpha)\to C_{\log}^{1}(\Delta_T^2;B_\alpha), \quad H\mapsto \hat\Lambda^{\mathsf s}_{S,1,\alpha}H:= R^{\mathsf s}(H),
\end{equation*}
satisfying
\begin{equation}
\label{eq:critical-right-inverse}
\hdd\hat\Lambda^{\mathsf s}_{S,1,\alpha}H=H.
\end{equation}
Moreover,
\begin{equation}
\label{eq:critical-log-estimate}
\norm{\hat\Lambda^{\mathsf s}_{S,1,\alpha}H}_{1,\log;\alpha}\le C_{T,S,C_\ast,\eta,\mathsf s} \norm{H}_{1,\alpha}.
\end{equation}
For an individual defect outside a fixed linear domain, the same estimate holds with the right-hand side replaced by $C_{T,S,C_\ast,\eta}(\norm{H}_{1,\alpha}+C_H)$.  Consequently, for every $0<\varepsilon<1$,
\begin{equation}
\label{eq:critical-subholder-estimate}
\norm{\hat\Lambda^{\mathsf s}_{S,1,\alpha}H}_{1-\varepsilon,\alpha}\le C_{\varepsilon,T,S,C_\ast,\eta,\mathsf s}\norm{H}_{1,\alpha}.
\end{equation}
\end{theorem}

\begin{proof}
The proof follows the same structure as Theorem~\ref{thm:subcritical-dyadic-sewing}.  We only record the modified dyadic recurrence.  With the notation used there, \eqref{eq:critical-scale-splitting-estimate} gives
\[
M_{n+1}\le \left(\frac12+C_\ast (T2^{-n})^\eta\right)M_n+C_H T2^{-n}.
\]
Set $N_n:=2^nM_n/T$.  Then
\[
N_{n+1}\le \left(1+2C_\ast T^\eta2^{-n\eta}\right)N_n+2C_H.
\]
Since $\sum_{n\ge0}2^{-n\eta}<\infty$ and $N_0=0$, the discrete Gronwall inequality gives
\[
N_n\le C_{T,C_\ast,\eta}C_H(1+n).
\]
Thus
\begin{equation}
\label{eq:critical-consecutive-bound}
M_n\le C_{T,C_\ast,\eta}C_H(1+n)T2^{-n}.
\end{equation}
If $r\simeq T2^{-n}$, then $1+n\lesssim 1+|\log r|$. Thus \eqref{eq:critical-consecutive-bound} gives, on consecutive dyadic intervals,
\[
\|R^{\mathsf s}_{t_{k+1}^n,t_k^n}\|_{B_\alpha}\le C C_H |t_{k+1}^n-t_k^n|\left(1+\left|\log |t_{k+1}^n-t_k^n|\right|\right).
\]
The dyadic decomposition argument from Theorem~\ref{thm:subcritical-dyadic-sewing} then gives
\[
\norm{R^{\mathsf s}_{ts}}_{B_\alpha}\le C\bigl(\norm{H}_{1,\alpha}+C_H\bigr)|t-s|\bigl(1+|\log |t-s||\bigr)
\]
on dyadic pairs.  The logarithmic modulus still tends to zero as $|t-s|\downarrow0$.  Repeating the endpoint-continuity argument in the proof of Theorem~\ref{thm:subcritical-dyadic-sewing}, with the above logarithmic bound in place of \eqref{eq:dyadic-holder-bound-scale}, gives a unique continuous extension to all pairs and preserves the estimate.  The algebraic identity $\hdd R^{\mathsf s}=H$ holds on dyadic triples by Lemma~\ref{lem:local-compatibility} and then extends by continuity.  The operator estimate follows from $C_H\le C_{\mathsf s}\|H\|_{1,\alpha}$ on the fixed critical selection domain. Finally,
\[
r(1+|\log r|)\le C_{\varepsilon,T}r^{1-\varepsilon}, \quad 0<r\le T,
\]
gives \eqref{eq:critical-subholder-estimate}.
\end{proof}

\begin{remark}
The actual representative constructed above depends on the admissible splitting rule.  If two admissible rules produce representatives $R$ and $\widetilde R$ for the same defect $H$, then $\hdd(R-\widetilde R)=0$.  Thus the remaining ambiguity is by semigroup cocycles, and the corresponding fixed-level quotient class is independent of the chosen representative.
\end{remark}

\section{Semigroup verification and applications}
\label{sec:semigroup-verification-and-applications}
We now specialize the abstract theory to the heat semigroup.  We first verify the scale-splitting assumptions, and then derive normalized semigroup increments and pathwise stochastic consequences.  The verification uses smoothing and frequency localization rather than invertibility, relying on the standard Sobolev spectral description of the heat semigroup~\cite{GT10}, elementary Littlewood--Paley estimates~\cite{BCD11} and a parabolic tail condition on the defect residuals.

\subsection{Sobolev verification of the scale splitting}
\label{sec:heat-verification}

Let $\T^d$ be the $d$-dimensional torus, let
\[
A_0:=1-\Delta,\quad S_t:=e^{-tA_0},
\]
and set
\[
B_\alpha:=H^\alpha(\T^d), \quad \norm{x}_{B_\alpha}:=\norm{A_0^{\alpha/2}x}_{L^2}.
\]
Let $P_{\le N}$ denote the $L^2$-orthogonal spectral projection of $A_0$ onto frequencies with spectral parameter $\lambda\lesssim N^2$, and set $P_{>N}:=\mathrm{Id}-P_{\le N}$. These projections commute with $S_t$ and are contractions on every $H^\alpha$.  The following standard estimates follow from the spectral theorem and the usual Littlewood--Paley cutoff bounds.

\begin{lemma}
\label{lem:heat-cutoff-estimates}
Let $0<\sigma<1$, $\beta>0$, and define $N_h:=h^{-\sigma/2}$ for $0<h\le1$.  Then there are constants $C,c>0$ such that for all $x\in H^{\alpha+\beta}$,
\begin{equation}
\label{eq:high-tail-estimate}
\norm{P_{>N_h}x}_{H^\alpha}
\le
C h^{\sigma\beta/2}
\norm{x}_{H^{\alpha+\beta}},
\end{equation}
and for all $x\in H^\alpha$,
\begin{equation}
\label{eq:low-inverse-estimate}
\norm{S_h^{-1}P_{\le N_h}x}_{H^\alpha}
\le
\exp(c h^{1-\sigma})
\norm{x}_{H^\alpha}.
\end{equation}
Here $S_h^{-1}P_{\le N_h}$ is defined spectrally on the finite-dimensional range of $P_{\le N_h}$.
\end{lemma}

\begin{proof}
For \eqref{eq:high-tail-estimate}, the spectral theorem gives
\[
\norm{P_{>N_h}x}_{H^\alpha}
\le C N_h^{-\beta}\norm{x}_{H^{\alpha+
\beta}}
=C h^{\sigma\beta/2}\norm{x}_{H^{\alpha+\beta}}.
\]
For \eqref{eq:low-inverse-estimate}, on the support of $P_{\le N_h}$ one has spectral parameter $\lambda\lesssim N_h^2=h^{-\sigma}$.  Therefore
\[
\norm{S_h^{-1}P_{\le N_h}}_{\mathcal L(H^\alpha)}
\le
\sup_{\lambda\lesssim h^{-\sigma}}e^{h\lambda}
\le
\exp(c h^{1-\sigma}).
\]
This proves the lemma.
\end{proof}

The next definition is defect-dependent.  It controls only the high-frequency tails of the residuals generated by $H$ at the relevant parabolic scales, rather than imposing a splitting condition on all of $H^\alpha$.

\begin{definition}
\label{def:parabolic-tail-condition}
Let $H\in Z\hC^{\gamma}(\Delta_T^3;B_\alpha)$ and let $\mathsf s$ be the heat spectral splitting rule described below.  We say that $H$ satisfies the {\bf parabolic tail condition} with parameters $(\beta,\sigma,K)$ if, for every dyadic parent interval $[a,c]$ with midpoint $b$, length $L=c-a$, $h=c-b=L/2$, and residual $Q_{ca}$ generated by the recursion, the vector $Y_{cba}:=Q_{ca}-H_{cba}$ satisfies
\begin{equation}
\label{eq:parabolic-tail-bound}
\norm{P_{>N_h}Y_{cba}}_{H^\alpha}\le K L^{\gamma+\rho},
\end{equation}
where $\rho:=\frac{\sigma\beta}{2}>0$. At the critical endpoint $\gamma=1$, the same condition is required with $L^{1+\rho}$ in \eqref{eq:parabolic-tail-bound}.
\end{definition}

\begin{remark}
The tail condition is automatic if the residuals have a little extra Sobolev
regularity.  Indeed, if for some $\beta>0$, $\|Y_{cba}\|_{H^{\alpha+\beta}}\le K L^\gamma$, then Lemma~\ref{lem:heat-cutoff-estimates} gives
\[
\|P_{>N_h}Y_{cba}\|_{H^\alpha}\le C h^{\sigma\beta/2}K L^\gamma \le C K L^{\gamma+\sigma\beta/2}.
\]
Thus the parabolic tail condition holds with $\rho=\sigma\beta/2$.
\end{remark}

We now use the parabolic cutoff to define the splitting.  For a parent residual $Q_{ca}$ and $0<h\le1$, set
\begin{equation}
\label{eq:heat-splitting-left}
Q_{ba}:=\frac12 S_h^{-1}P_{\le N_h}Y_{cba},\quad Q_{cb}:=\frac12 P_{\le N_h}Y_{cba}+P_{>N_h}Y_{cba}.
\end{equation}
Then $Q_{cb}+S_hQ_{ba}=Y_{cba}$, and hence the splitting identity holds.  In the sequel we state the estimate for $T\le1$; general finite horizons follow by
subdivision.

\begin{lemma}
\label{lem:heat-splitting-identity}
The splitting \eqref{eq:heat-splitting-left} satisfies
\begin{equation*}
Q_{ca}=Q_{cb}+S_hQ_{ba}+H_{cba}.
\end{equation*}
\end{lemma}

\begin{proof}
By construction,
\[
Q_{ca}=Y_{cba}+H_{cba}=\frac12P_{\le N_h}Y_{cba}+P_{>N_h}Y_{cba}+\frac12P_{\le N_h}Y_{cba}+H_{cba}=Q_{cb}+S_hQ_{ba}+H_{cba}.\qedhere
\]
\end{proof}

The required estimate for the heat spectral splitting is as follows.

\begin{proposition}
\label{prop:heat-spectral-splitting-estimate}
Let $0<T\le1$, $0<\sigma<1$, and suppose that $H$ satisfies the parabolic tail condition \eqref{eq:parabolic-tail-bound}.  Then the heat spectral splitting satisfies
\begin{equation}
\label{eq:heat-critical-estimate}
\max\{\norm{Q_{cb}}_{H^\alpha},\norm{Q_{ba}}_{H^\alpha}\}\le \left(\frac12+C L^\eta\right)\norm{Q_{ca}}_{H^\alpha} +C\bigl(\norm{H}_{\gamma,\alpha}+K\bigr)L^\gamma,
\end{equation}
where $\eta:=\min\{1-\sigma,\rho\}>0$ and $\rho=\frac{\sigma\beta}{2}$.
In particular, at $\gamma=1$ this is a critical admissible scale splitting.  If $0<\gamma<1$, then after restricting to a sufficiently small time horizon $T_0$, it is a subcritical admissible scale splitting, because $1/2+C T_0^\eta<2^{-\gamma}$.
\end{proposition}

\begin{proof}
We estimate $\norm{Q_{ba}}_{H^\alpha}$ and $\norm{Q_{cb}}_{H^\alpha}$ separately. For $\norm{Q_{ba}}_{H^\alpha}$, Lemma~\ref{lem:heat-cutoff-estimates} gives
\[
\norm{Q_{ba}}_{H^\alpha}\le\frac12 e^{c h^{1-\sigma}}\norm{P_{\le N_h}Y_{cba}}_{H^\alpha}\le\frac12(1+C L^{1-\sigma})\norm{Y_{cba}}_{H^\alpha}.
\]
Since $Y_{cba}=Q_{ca}-H_{cba}$ and $\norm{H_{cba}}_{H^\alpha}\le C\norm{H}_{\gamma,\alpha}L^\gamma$, we get
\begin{equation}
\label{eq:left-heat-estimate}
\norm{Q_{ba}}_{H^\alpha}\le\left(\frac12+C L^{1-\sigma}\right)\norm{Q_{ca}}_{H^\alpha}+C\norm{H}_{\gamma,\alpha}L^\gamma.
\end{equation}
For $\norm{Q_{cb}}_{H^\alpha}$, using \eqref{eq:heat-splitting-left},
\[
\norm{Q_{cb}}_{H^\alpha}
\le
\frac12
\norm{P_{\le N_h}Y_{cba}}_{H^\alpha}
+
\norm{P_{>N_h}Y_{cba}}_{H^\alpha}.
\]
The first term is bounded by
\[
\frac12\norm{Y_{cba}}_{H^\alpha}
\le
\frac12\norm{Q_{ca}}_{H^\alpha}
+C\norm{H}_{\gamma,\alpha}L^\gamma.
\]
The second term is bounded by 
\[
\norm{P_{>N_h}Y_{cba}}_{H^\alpha}\le K L^{\gamma+\rho}\le K T^\rho L^\gamma, \quad L\le T\le1.
\]
Combining these estimates gives
\begin{equation}
\label{eq:right-heat-estimate}
\norm{Q_{cb}}_{H^\alpha}
\le
\frac12
\norm{Q_{ca}}_{H^\alpha}
+C(\norm{H}_{\gamma,
\alpha}+K)L^\gamma.
\end{equation}
Together \eqref{eq:left-heat-estimate} and \eqref{eq:right-heat-estimate} prove \eqref{eq:heat-critical-estimate}, with $\eta=\min\{1-\sigma,\rho\}$.  The last statement follows because $2^{-\gamma}>1/2$ for $0<\gamma<1$.
\end{proof}

Combining the algebraic splitting with the heat semigroup estimates gives the following representative result.

\begin{theorem}
\label{thm:heat-sewing-sobolev}
Let $S_t=e^{-t(1-\Delta)}$ on $H^\alpha(\T^d)$, and let $H\in Z\hC^{\gamma}(\Delta_T^3;B_\alpha)$ satisfy the parabolic residual tail condition of
Definition~\ref{def:parabolic-tail-condition}.  Assume also that the residual continuity criterion \eqref{eq:residual-continuity-criterion} holds for the heat spectral rule.  Then the heat spectral splitting produces a semigroup Sewing representative
\[
\hat\Lambda^{\mathsf{heat}}_{S,\gamma,\alpha}H:=R^{\mathsf s,\mathrm{heat}}(H).
\]
If $0<\gamma<1$, then on sufficiently short time intervals one has
\[
\hat\Lambda^{\mathsf{heat}}_{S,\gamma,\alpha}H\in C^{\gamma}(\Delta_T^2;B_\alpha), \quad\hdd\hat\Lambda^{\mathsf{heat}}_{S,\gamma,\alpha}H=H,
\]
and 
\[
\norm{\hat\Lambda^{\mathsf{heat}}_{S,\gamma,\alpha}H}_{\gamma,\alpha}\le C_T\bigl(\norm{H}_{\gamma,\alpha}+K\bigr).
\]
At the endpoint $\gamma=1$,
\[
\hat\Lambda^{\mathsf{heat}}_{S,1,\alpha}H\in C_{\log}^{1}(\Delta_T^2;B_\alpha), \quad\hdd\hat\Lambda^{\mathsf{heat}}_{S,1,\alpha}H=H,
\]
and
\[
\norm{\hat\Lambda^{\mathsf{heat}}_{S,1,\alpha}H}_{1,\log;\alpha}\le C_T\bigl(\norm{H}_{1,\alpha}+K\bigr).
\]
On any linear class where $K\le C\norm{H}_{\gamma,\alpha}$ and the residual continuity criterion holds uniformly, this construction defines a continuous linear Sewing map.
\end{theorem}

\begin{proof}
Lemma~\ref{lem:heat-splitting-identity} gives the exact splitting identity, and Proposition~\ref{prop:heat-spectral-splitting-estimate} gives the admissible scale estimate.  Hence Theorem~\ref{thm:critical-dyadic-sewing} applies directly when $\gamma=1$.
For $0<\gamma<1$, the coefficient
$1/2+CL^\eta$ in
Proposition~\ref{prop:heat-spectral-splitting-estimate} is smaller than $2^{-\gamma}$ on sufficiently small dyadic scales.  Thus Theorem~\ref{thm:subcritical-dyadic-sewing} applies after finitely many initial levels, which only affect the constant.  The estimate and the identity
$\hdd\hat\Lambda^{\mathsf{heat}}_{S,\gamma,\alpha}H=H$ then follow by the usual dyadic decomposition and continuity argument.  Linearity and continuity follow
from the fixed linear splitting rule and the uniform bound $K\le C\norm{H}_{\gamma,\alpha}$.
\end{proof}

\begin{remark}
Theorem~\ref{thm:heat-sewing-sobolev} should be read as an actual representative result on a heat-admissible class of defects, not as a Sewing map on the whole space $Z\hC^{\gamma}(\Delta_T^3;B_\alpha)$.  For an individual defect the estimate depends on the tail constant $K$.  A continuous linear Sewing map is obtained only on linear classes where this tail constant is controlled by $\norm{H}_{\gamma,\alpha}$.
\end{remark}

\subsection{Normalized semigroup increments}
\label{sec:convolution-primitives}
We now explain how a scale-dependent semigroup Sewing representative produces a normalized semigroup increment. This is the form needed for mild evolution
equations.

Let $0<\gamma\le1$ and let $A\in C(\Delta_T^2;B_\alpha)$ be such that $\hdd A\in Z\hC^{\gamma}(\Delta_T^3;B_\alpha)$. Assume moreover that $\hdd A$ belongs to an admissible scale-dependent selection domain, and fix the corresponding admissible splitting rule $\mathsf s$.  Define
\begin{equation}
\label{eq:sewn-exact-part}
\mathcal J^{\mathsf s}_{\gamma,\alpha}(A)
:=
A-
\hat\Lambda^{\mathsf s}_{S,\gamma,\alpha}(\hdd A),
\end{equation}
where for $\gamma=1$ the endpoint map from Theorem~\ref{thm:critical-dyadic-sewing} is used.

\begin{proposition}
\label{prop:sewn-exact-cocycle}
The increment $\mathcal J^{\mathsf s}_{\gamma,\alpha}(A)$ defined above satisfies $\hdd\mathcal J^{\mathsf s}_{\gamma,\alpha}(A)=0$.
\end{proposition}

\begin{proof}
By \eqref{eq:sewn-exact-part} and the right-inverse property,
\[
\hdd\mathcal J^{\mathsf s}_{\gamma,\alpha}(A)=\hdd A-\hdd\hat\Lambda^{\mathsf s}_{S,\gamma,\alpha}(\hdd A)=0.\qedhere
\]
\end{proof}

The following result identifies the corrected cocycle as the semigroup increment of a uniquely determined normalized semigroup increment.

\begin{theorem}
\label{thm:normalized-convolution-primitive-revised}
Let $0<\gamma\le1$, let $A\in C(\Delta_T^2;B_\alpha)$, and assume that $\hdd A\in Z\hC^{\gamma}(\Delta_T^3;B_\alpha)$ belongs to the admissible selection domain of the chosen rule $\mathsf s$.  Define a path $I^{\mathsf s}_{\gamma,\alpha}(A)_t:=\mathcal J^{\mathsf s}_{\gamma,\alpha}(A)_{t0}$. Then
\begin{equation}
\label{eq:convolution-primitive-identity}
\hdd I^{\mathsf s}_{\gamma,\alpha}(A)
=
\mathcal J^{\mathsf s}_{\gamma,\alpha}(A).
\end{equation}
Equivalently,
\begin{equation*}
(\hdd I^{\mathsf s}_{\gamma,\alpha}(A))_{ts}=A_{ts}-\bigl(\hat\Lambda^{\mathsf s}_{S,\gamma,\alpha}(\hdd A)\bigr)_{ts}.
\end{equation*}
The path is the unique path satisfying \eqref{eq:convolution-primitive-identity} and $I^{\mathsf s}_{\gamma,\alpha}(A)_0=0$.
\end{theorem}

\begin{proof}
Set $\mathcal J:=\mathcal J^{\mathsf s}_{\gamma,\alpha}(A)$ and $I^{\mathsf s}:=I^{\mathsf s}_{\gamma,\alpha}(A)$.  By Proposition~\ref{prop:sewn-exact-cocycle}, $\hdd\mathcal J=0$.  Evaluating $\hdd\mathcal J=0$ at $(t,s,0)$ gives
\[
\mathcal J_{t0}-\mathcal J_{ts}-S_{t-s}\mathcal J_{s0}=0.
\]
Using $I_t^{\mathsf s}=\mathcal J_{t0}$, we get
\[
(\hdd I^{\mathsf s})_{ts}=I^{\mathsf s}_t-S_{t-s}I^{\mathsf s}_s=\mathcal J_{ts}.
\]
This proves \eqref{eq:convolution-primitive-identity}.  The expanded identity is just the definition of $\mathcal J$.  If $J$ is another path with $J_0=0$ and $\hdd J=\mathcal J$, then $\hdd(J-I^{\mathsf s})=0$.  Evaluating at $(t,0)$ gives
\[
J_t-I^{\mathsf s}_t=S_t(J_0-I^{\mathsf s}_0)=0.
\]
Thus the semigroup increment is unique under the initial condition.
\end{proof}

\subsection{Stochastic consequences}
\label{sec:stochastic-consequences}
This subsection records pathwise stochastic consequences. Stochastic arguments provide the required defect and tail estimates, while the Sewing step itself is deterministic.  Let $0<\gamma\le1$.  A random semigroup defect of regularity $\gamma$ is a measurable map
\[
H:\Omega\to Z\hC^{\gamma}(\Delta_T^3;B_\alpha).
\]
We write $H\in L^p\big(\Omega;Z\hC^{\gamma}(\Delta_T^3;B_\alpha)\big)$ if $\mathbb E\bigl[\norm{H}_{\gamma,\alpha}^p\bigr]<\infty$.

Once the random defect satisfies the deterministic admissibility assumptions almost surely, the Sewing construction applies pathwise and the deterministic estimate yields the corresponding $L^p(\Omega)$ bound.

\begin{theorem}
\label{thm:stochastic-sewing-estimates}
Fix an admissible scale-dependent selection rule $\mathsf s$ on a domain $\D^{\mathsf s}_{\gamma}(\Delta_T^3;B_\alpha)$.  Let $p\in[1,\infty)$.
If $0<\gamma<1$ and $H\in L^p\big(\Omega; \D^{\mathsf s}_{\gamma}(\Delta_T^3;B_\alpha)\big)$, then $R^{\mathsf s}:=\hat\Lambda^{\mathsf s}_{S,\gamma,\alpha}H$ satisfies
\begin{equation}
\label{eq:stochastic-subcritical-estimate}
\mathbb E\left[\norm{R^{\mathsf s}}_{\gamma,\alpha}^p\right]\le C^p \mathbb E\left[\norm{H}_{\gamma,\alpha}^p\right]
\end{equation}
and $\hdd R^{\mathsf s}=H$ almost surely.  For a family of individual random defects not lying in one fixed linear domain, the same estimate holds with the right-hand side replaced by
\[
C^p\mathbb E\left[\bigl(\norm{H}_{\gamma,\alpha}+C_H(\omega)\bigr)^p\right].
\]
If $H\in L^p\big(\Omega;\D^{\mathsf s}_{1}(\Delta_T^3;B_\alpha)\big)$, then $R^{\mathsf s}:=\hat\Lambda^{\mathsf s}_{S,1,\alpha}H$ satisfies
\begin{equation}
\label{eq:stochastic-critical-estimate}
\mathbb E\left[\norm{R^{\mathsf s}}_{1,\log;\alpha}^p\right]\le C^p\mathbb E\left[\norm{H}_{1,\alpha}^p\right]
\end{equation}
and $\hdd R^{\mathsf s}=H$ almost surely.  Again, for individual defects one may replace $\norm{H}_{1,\alpha}$ by $\norm{H}_{1,\alpha}+C_H(\omega)$.
\end{theorem}

\begin{proof}
The deterministic maps in Theorems~\ref{thm:subcritical-dyadic-sewing} and~\ref{thm:critical-dyadic-sewing} are linear and continuous once the splitting rule is fixed.  Applying the deterministic estimates pathwise and integrating gives \eqref{eq:stochastic-subcritical-estimate} and \eqref{eq:stochastic-critical-estimate}.  The identities $\hdd R^{\mathsf s}=H$ hold pathwise and therefore almost surely.
\end{proof}

The pathwise construction is also stable under approximation of the random defect.

\begin{corollary}
\label{cor:stochastic-stability}
Let $0<\gamma<1$ and suppose
\[
H^n,H\in L^p\big(\Omega; \D^{\mathsf s}_{\gamma}(\Delta_T^3;B_\alpha)\big), \quad \norm{H^n-H}_{L^p\big(\Omega;C^{\gamma}(\Delta_T^3;B_\alpha)\big)}\to0,
\]
Then
\[
\norm{\hat\Lambda^{\mathsf s}_{S,\gamma,\alpha}H^n-
\hat\Lambda^{\mathsf s}_{S,\gamma,\alpha}H}_{L^p\big(\Omega;C^{\gamma}(\Delta_T^2;B_\alpha)\big)}\to0.
\]
At the endpoint $\gamma=1$, convergence in $L^p\big(\Omega;C^{1}(\Delta_T^3;B_\alpha)\big)$ and convergence of the splitting constants imply convergence in $L^p\big(\Omega; C_{\log}^{1}(\Delta_T^2;B_\alpha)\big)$.
\end{corollary}

\begin{proof}
The claim follows from the continuity and linearity of the deterministic Sewing map on the fixed admissible domain.  In an individual-defect formulation, the same conclusion follows if the corresponding individual constants, for instance parabolic tail constants in the heat case, converge in $L^p(\Omega)$.
\end{proof}

\begin{remark}
The preceding results do not introduce a separate stochastic Sewing mechanism. The construction is applied pathwise to $H(\omega)$.  Stochastic estimates are
used only to verify the admissibility assumptions and the required $L^p(\Omega)$ bounds on the defect and, in the heat case, on the residual tail constant.
\end{remark}

\vskip 0.1in

\noindent
{\bf Acknowledgments.} This work is supported by the National Natural Science Foundation of China (12571019), the Natural Science Foundation of Gansu Province (25JRRA644) and Innovative Fundamental Research Group Project of Gansu Province (23JRRA684).

\noindent
{\bf Declaration of interests. } The authors have no conflicts of interest to disclose.

\noindent
{\bf Data availability. } Data sharing is not applicable as no new data were created or analyzed.

\medskip 


\end{document}